\documentclass[10pt,a4paper]{article}
\usepackage[numbers,sort&compress]{natbib}
\usepackage[T1]{fontenc}
\usepackage[utf8]{inputenc}
\usepackage{authblk}
\usepackage{amsmath,amssymb,amsthm,esint,bm}%,mathrsfs
\usepackage{mathrsfs}
\usepackage{bookmark}
\usepackage{amsmath}
\usepackage{enumitem}
\allowdisplaybreaks[4]

\newtheorem{definition}{Definition}[section]
\newtheorem{theorem}[definition]{Theorem}
\newtheorem{lemma}[definition]{Lemma}
\newtheorem{proposition}[definition]{Proposition}
\newtheorem{corollary}[definition]{Corollary}
\theoremstyle{remark}
\newtheorem{remark}[definition]{Remark}
\numberwithin{equation}{section}

\newcommand{\abs}[1]{\lvert#1\rvert}
\newcommand{\Abs}[1]{\left\lvert#1\right\rvert}

\renewcommand{\r}{{\mathbb{R}}}
\newcommand{\rn}{{\mathbb{R}^d}}

\allowdisplaybreaks[4]

\title{Sharp regularity for a class of degenerate/singular fully nonlinear elliptic equations with Hamiltonian terms}

\author[a]{Wentao Huo}
\author[a]{Xiaofeng Jin}
\author[b]{Lingwei Ma}
\author[b]{Zhenqiu Zhang\thanks{Corresponding author.}}
\affil[a]{School of Mathematical Sciences, Nankai University, Tianjin 300071, P.R. China}
\affil[b]{School of Mathematical Sciences and LPMC, Nankai University, Tianjin 300071, P.R. China}
\date{\today}
\usepackage{hyperref}
\begin{document}
\maketitle
\footnotetext[1]{E-mail: huowentaoouc@163.com (W. Huo), 1120220040@mail.nankai.edu.cn (X. Jin), malingwei@nankai.edu.cn (L. Ma), zqzhang@nankai.edu.cn (Z. Zhang).}

\begin{abstract}
We investigate the regularity of the  viscosity solutions to a class of degenerate/singular fully nonlinear elliptic equations with Hamiltonian terms. To overcome the difficulty caused by the simultaneous presence of the general degenerate/singular gradient terms and Hamiltonian terms, we  analyze the coupled interplay between the degeneracy/singularity law and the growth of Hamiltonian terms and establish lower regularity results. Finally, we obtain sharp interior $C^{1,\alpha}$ regularity estimates via a geometric tangential method.



\end{abstract}
%\renewcommand{\thepage}{\roman{page}}
%\setcounter{page}{1}

%%\tableofcontents

\section{Introduction}\label{section1}
In this paper, we consider the following second order degenerate/singular fully nonlinear elliptic equations with Hamiltonian terms:
\begin{equation}\label{model}
\Phi(\abs{Du}, x)F(D^2 u, x)+H({Du}, x)=f(x) \quad  \text{in} \quad \Omega ,
\end{equation}
where $\Omega \subset \mathbb{R}^d$ $(d\geq 2)$ is an open and bounded domain. Throughout this paper we assume that $F$, $\Phi$, $f$ and $h$ satisfy the following hypotheses:
\begin{enumerate}[label=(\text{A}\arabic{enumi}),ref=\textbf{A}\arabic{enumi}]
	\item \label{A1} The fully nonlinear operator $F:S^{d}\times \Omega\rightarrow \mathbb{R}$ is continuous and uniformly $(\lambda,\Lambda)$-elliptic in the sense that
	$$\lambda\|N\|\leq F(M+N,x)-F(M,x)\leq \Lambda\|N\|$$
	for some $0<\lambda\leq \Lambda$ and each $M,N\in S^{d}$ with $N\geq 0$. Here $S^{d}$ stands for
	the set of all $d\times d$ real symmetric matrices. Without loss of generality, we assume that $F(0,x) = 0$ for all $x\in \Omega$.
	\item \label{A2}
	We assume a uniform continuity assumption on the coefficients of $F$, namely, there exist constants $C>0$ and $\theta\in(0,1)$ such that
	\begin{equation*}%\label{lianxu}
		{\rm osc}_{F}(x,y):=\sup\limits_{M\in S^{d}\setminus \{0\}}\frac{\Abs{F(M,x)-F(M,y)}}{\|M\|}\leq C|x-y|^{\theta}
	\end{equation*}
	for all $x,y\in \Omega$. For simplicity purposes, we shall often write ${\rm osc}_{F}(x):={\rm osc}_{F}(x,0)$. Moreover, for notation purposes, we denote
	$$C_{F}:=\inf\{C>0:{\rm osc}_{F}(x,y)\leq C|x-y|^{\theta},\;\forall x,y\in \Omega\}.$$	
	\item \label{A3} The function $\Phi:[0,\infty)\times\Omega\rightarrow [0,\infty)$ is a continuous map satisfying the following properties:\\
	(i) there exist constants $s(\Phi)\geq i(\Phi)>-1$ such that the map $t\mapsto \frac{\Phi(t,x)}{t^{i(\Phi)}}$ is almost non-decreasing with constant $L\geq 1$ in $(0,\infty)$ in the sense that
	$$\frac{\Phi(t,x)}{t^{i(\Phi)}}\leq L \frac{\Phi(s,x)}{s^{i(\Phi)}}\quad {\rm whenever \;\;}0<t\leq s<\infty\;{\rm and\;} x\in \Omega,$$
	and the map $t\mapsto \frac{\Phi(t,x)}{t^{s(\Phi)}}$ is almost non-increasing  with constant $L\geq 1$ in $(0,\infty)$ in the sense that
	$$L\frac{\Phi(t,x)}{t^{s(\Phi)}}\geq \frac{\Phi(s,x)}{s^{s(\Phi)}}\quad {\rm whenever \;\;}0<t\leq s<\infty\;{\rm and\;} x\in \Omega;$$	
	(ii) there exist constants $0<\nu_{0}\leq \nu_{1}$ such that $\nu_{0}\leq \Phi(1,x)\leq \nu_{1}$ for all $x\in \Omega$.
	\item \label{A4} The Hamiltonian
	term $H:\mathbb{R}^{d} \times \Omega\rightarrow \mathbb{R}$ is continuous and there exist constants $\mathcal{K},\mathcal{M}>0$ and $0<m\leq 1+i(\Phi)$ such
	that
	\begin{equation}
		|H(t,x)|\leq \mathcal{K}+\mathcal{M}|t|^{m}
	\end{equation}
for every $t\in\mathbb{R}^{d}$, $x\in\Omega$.	
	\item \label{A5} The source term $f$ belongs to $C({\Omega}) \cap L^{\infty}(\Omega)$.
\end{enumerate}

Equations of the form \eqref{main} were introduced by Lions and Lasry in \cite{Lions1985,Lions1989} for the case $F=\Delta$, and later extended to the fully nonlinear setting by Birindelli and Demengel in \cite{Birindelli2004,Birindelli2007CPAA}. These equations often appear in fields such as stochastic games \cite{Attouchi20147,Banerjee2020}, stochastic optimal control problems \cite{Fleming,Leoni}, and image enhancement \cite{Bronzi2020}. The existence and uniqueness of viscosity solutions to such equations were established, see for instance \cite{Birindelli2006,Birindelli2007CPAA,Silva2023}.
Much attention has been devoted to exploring the regularity and qualitative properties of solutions to this class of partial differential equations in the last decades, please refer to \cite{Birindelli2010JDE,Birindelli2012NON,Nornberg2019,Silva-Nornberg2021,Davil2009, Davil2010, Imbert2011JDE,Junges2010,Byun2024,Byun2025} and references therein. Our present paper aims to study interior  H\"{o}lder regularity estimates for viscosity solutions to \eqref{model}.

The development of  H\"{o}lder regularity theory for viscosity solutions of fully nonlinear elliptic equations represents a central achievement in the theory of non-divergence form PDEs. This topic began with Krylov and Safonov's groundbreaking work \cite{Krylov1979,Krylov1980} on linear non-divergence form equations. They established Harnack inequalities and interior H\"{o}lder estimates by means of measure-theoretic arguments. Subsequently, with the viscosity solutions framework established, Caffarelli proved a series of regularity results for viscosity solutions to fully nonlinear equations $F(D^{2}u,x)=f(x)$ in the seminal work \cite{Caffarelli1989} (see also monograph \cite{Caff1} for more details).

In the singular/degenerate fully nonlinear setting, the most celebrated prototype is
\begin{equation}\label{amodel}
	\abs{Du}^{p}F(D^2 u)=f(x) \quad  \text{in} \quad \Omega
\end{equation}
with $p>-1$. We mention the seminal work of Imbert and Silvestre \cite{Imbert1}, where they considered the degenerate case $(p>0)$ and resorted to an improvement-of-flatness approach to prove that viscosity solutions of \eqref{amodel} are locally $C^{1,\alpha}$. Further developments on the regularity theory for related models have been reported in \cite{Ricarte}. Specifically, Ara$\acute{\rm u}$jo, Ricarte and Teixeira investigated the following degenerate elliptic equations with variable coefficient in the unit ball $B_{1}$:
\begin{equation}\label{abmodel}
	\Phi(\abs{Du},x)F(D^2 u,x)=f(x),
\end{equation}
where $f\in L^{\infty}(B_{1})$ and $\Phi:\rn\times B_{1}\rightarrow \r$ degenerate as $\Phi(\abs{Du},x)\sim \abs{Du}^{p}$ for some $p>0$.
Under a suitable uniform continuity assumption on the coefficients of $F$, they conducted a geometric analysis unveiling the optimal regularity of solutions to \eqref{abmodel}, that is,  solutions are shown to be of class $C_{loc}^{1,\alpha}$, for $\alpha=\min\left\{\alpha_{0},\frac{1}{1+p}\right\}$, where $\alpha_{0}$ is the H\"{o}lder exponent coming from the Krylov–Safonov regularity for equation $F(D^{2}u)=0$ (see \cite[Chapter 5]{Caff1}). Since then, these types of regularity results have been extended to various kinds of degenerate/singular fully nonlinear elliptic equations (see \cite{Bronzi2020,Fili, Silva2020,Fang, Silva2023}).  It is noteworthy to mention that the recent paper \cite{Baasandorj} considered the general singular/degenerate operator $\Phi$ satisfying the assumption \eqref{A3} and derived local optimal $C^{1,\alpha}$ regularity of solutions to \eqref{abmodel}.
 
 As to fully nonlinear equations with Hamiltonian terms, Birindelli-Demengel's key works \cite{Birindelli2014ESAIM, Birindelli2015, B-Demengel2016, B-Demengel2019} studied the regularity of solutions to the following fully nonlinear equations
\begin{equation*}
	\abs{Du}^{p}F(D^2 u)+h(x)\abs{Du}^{m} =f(x) \quad  \text{in} \quad \Omega,
\end{equation*}
where $p>-1$, $0<m\leq p+2$, and $f\in C(\Omega)\cap L^{\infty}(\Omega)$. To be precise, they obtained the local $C^{1,\alpha}$ regularity for a universal constant $\alpha\in(0,1)$ in \cite{Birindelli2014ESAIM, Birindelli2015, B-Demengel2016} for the sublinear and linear cases $0<m\leq p+1$, and in \cite{B-Demengel2019} for the superlinear case $p+1<m\leq p+2$. Very recently, Andrade and Nascimento \cite{Andrade} considered the degenerate fully nonlinear equations with Hamiltonian terms of the type
\begin{equation*}
	\abs{Du}^{p}F(D^2 u, x)+h(x)\abs{Du}^{m} =f(x) \quad  \text{in} \quad B_{1},
\end{equation*}
where $p>0$, $0<m\leq 1+p$, and $f,h\in C(B_{1})\cap L^{\infty}(B_{1})$. Under the suitable continuity assumption on the coefficients of $F$, they established the optimal interior $C^{1,\alpha}$ regularity estimates with $\alpha\in (0,\alpha_{0})\cap (0,\frac{1}{1+p}]$ via utilizing perturbation techniques.

To the best of our knowledge, there remains a notable absence of regularity theory for solutions of fully nonlinear elliptic equations simultaneously involving Hamiltonian terms and a quite general degeneracy/singularity law of the type $\Phi(|Du|,\cdot)$. In this paper, our primary focus is on establishing optimal interior H\"{o}lder gradient estimates of solutions to equation \eqref{model}. These estimates not only encompass the $p$-growth and double-phase growth but also include the variable exponent, log-type and Orlicz double-phase growth cases, undoubtedly constituting a valuable addition to the regularity theory of fully nonlinear PDEs.

We now state the main result of this paper.
\begin{theorem}\label{main}
	 Assume that the hypotheses \eqref{A1}-\eqref{A5} hold. Let $u\in C(\Omega)$ be a viscosity solution of \eqref{model} and $\alpha$ be chosen to satisfy
	\begin{equation}\label{exponent}
		\alpha\in \left\{
		\begin{array}{lcl}
			(0,\alpha_{0})\cap \left(0,\frac{1}{1+s(\Phi)}\right] & \text{if} & i(\Phi)\geq 0, \\
			(0,\alpha_{0})\cap \left(0,\frac{1}{1+s(\Phi)-i(\Phi)}\right] &\text{if}&  -1<i(\Phi)<0.
		\end{array}
		\right.
	\end{equation}
	Then $u\in C_{loc}^{1,\alpha}(\Omega)$. More precisely, for any subdomain $\Omega^{\prime}\subset\subset\Omega$, there holds
	\begin{itemize}
		\item [{\rm$({{\rm i}})$}] if $0<m<1+i(\Phi)$, then
	\begin{equation}\label{giji1}
	[u]_{C^{1, \alpha}({\Omega^{\prime}})}\leq C\left(1+\|u\|_{L^{\infty}\left(\Omega\right)}+\left(\frac{\|f\|_{L^{\infty}\left(\Omega\right)}+\mathcal{K}}{\nu_{0}}\right)^{\frac{1}{1+i(\Phi)}}+\left(\frac{\mathcal{M}}{\nu_{0}}\right)^{\frac{1}{1+i(\Phi)-m}}\right),
   \end{equation}
where the constant $C$ depends on $d,\lambda,\Lambda,\alpha,m,\theta,L,i(\Phi),C_{F}$ and ${\rm dist}(\Omega,\Omega^{\prime})$;
		\item [{\rm$({{\rm ii}})$}] if $m=1+i(\Phi)$, then
			\begin{equation}\label{giji2}
			[u]_{C^{1, \alpha}({\Omega^{\prime}})}\leq C\left(1+\|u\|_{L^{\infty}(\Omega)}\right),
		\end{equation}
	where the constant $C$ depends in addition on $\nu_{0}$, $\|f\|_{L^{\infty}\left(\Omega\right)}$, $\mathcal{K}$ and $\mathcal{M}$.
	\end{itemize}
\end{theorem}
\begin{remark}
	The outcome of Theorem \ref{main} is sharp, in the light of the analysis made by a scaling argument in \cite[Section 3]{Andrade}. Note that our Theorem \ref{main} concerns optimal interior regularity, unlike \cite{Birindelli2014ESAIM, B-Demengel2016} which only proves that solutions belong to $C^{1,\alpha}$ for some unspecified exponent $\alpha\in(0,1)$.	In addition, due to the generality of degeneracy/singularity law and Hamiltonian term, our finding above can be regarded as a extension and complement of \cite{Andrade}, which only deals with the case of single power-type degeneracy rate and a special type Hamiltonian term. It is worth emphasizing that this generalization of regularity results here is nontrivial. In fact, due to the simultaneous presence of the abstract forms of $\Phi$ and $H$, implementing the strategy from \cite{Andrade,Baasandorj,Ricarte} becomes a delicate task. In particular, we need to analyze the coupled interplay between the Hamiltonian term and the general degenerate/singular gradient term, and establish new compactness for solutions as well as a new approximation lemma. 
\end{remark}
\begin{remark}
	It is worth pointing out that our result can be applied to some remarkable
	cases besides $\Phi(|Du|,x)=|Du|^{p}+a(x)|Du|^{q}$ or $\Phi(|Du|,x)=|Du|^{p(x)}+a(x)|Du|^{q(x)}$, such as
	\begin{itemize}	
		\item  [{\rm$\bullet$}] $\Phi(|Du|,x)=|Du|^{p}+a(x)|Du|^{p}\log(|Du|+1)$ for $p>-1$ and $0\leq a(\cdot)\in C(\Omega)$;
		\item  [{\rm$\bullet$}] $\Phi(|Du|,x)=\phi(|Du|)+a(x)\varphi(|Du|)$ for suitable $N$-functions (cf. \cite{Adams2003}) $\phi,\varphi$ and $0\leq a(\cdot)\in C(\Omega)$.
	\end{itemize}
Furthermore, condition \eqref{A4} on Hamiltonian term includes some typical examples, such as
\begin{itemize}	
	\item  [{\rm$\bullet$}] $H(t,x)=h(x)|t|^{m}$ for $m>0$ and function $h\in C(\Omega)\cap L^{\infty}(\Omega)$;
	\item  [{\rm$\bullet$}] $H(t,x)=\left\langle h(x),t \right \rangle |t|^{m-1}$ for $m>0$ and the vector field $h\in C(\Omega,\rn)\cap L^{\infty}(\Omega,\rn)$;
	\item  [{\rm$\bullet$}] $H(t,x)=\sum\limits_{i=1}^{N}h_{i}(x)|t|^{m_{i}}$ for $m_{i}>0$, function $h_{i}\in C(\Omega)\cap L^{\infty}(\Omega)$, $i=1,2,\cdots,N$, and $m=\max\left\{m_{1},m_{2},\cdots,m_{N}\right\}.$
\end{itemize}
\end{remark}

Since viscosity solutions to convex/concave equations $F(D^{2}u)=0$ are locally of class $C^{1,1}$ by classical  Evans–Krylov theory \cite{Evans, Krylov1}, an important consequence of Theorem \ref{main} is the following optimal regularity result.
\begin{corollary}\label{coro1}
Assume that the assumptions of Theorem \ref{main} are in force. Suppose further that operator $F$ is convex (or concave). There holds
	\begin{itemize}
	\item [{\rm$({{\rm i}})$}] if $i(\Phi)\geq 0$, then $u\in C_{loc}^{1,\frac{1}{1+s(\Phi)}}(\Omega)$;
	\item [{\rm$({{\rm ii}})$}] if  $-1<i(\Phi)<0$, then $u\in C_{loc}^{1,\frac{1}{1+s(\Phi)-i(\Phi)}}(\Omega)$.
\end{itemize}
\end{corollary}

The remainder of this paper is organized as follows. In Section \ref{section2}, we introduce the basic notions and some well-known results, and then explain how to reduce the problem to a smallness regime. In Section \ref{section3}, we obtain the H\"{o}lder and Lipschitz regularity results, which aim at producing compactness of solutions to auxiliary problems appearing further in our arguments. Section \ref{section4} establishes a geometric tangential path, which is one of the main ingredients in the realm of regularity transmission. In the last section, we provide a detailed proof of Theorem \ref{main}.
%%%%%%%%%%%%%%%%%%%%%%%%%%%%%%%%%%%%%%%%%%%%%%%%%%%%%%%%
%%%%%%%%%%%%%%%%%%%%%%%%%%%%%%%%%%%%%%%%%%%%%%%%%%%%%%%%
\section{Preliminaries}\label{section2}
\subsection{Notations and basic concepts}
Throughout this paper, let $B_{r}(x_{0})$ be the open ball with radius $r$ and centred at $x_{0}\in\rn$. If not important, or clear from the context, we will omit indicating the centre by writing $B_{r}:=B_{r}(x_{0})$. In particular, we shall simply denote $B_{1}:=B_{1}(0)$. In what follows, $C$ denotes a constant whose value may vary from line to line, and only the relevant dependencies are specified in parentheses.

We begin with the definition of the Pucci extremal operators.
\begin{definition}[\bf Pucci extremal operators]
	Let $0<\lambda\leq \Lambda$. For any $M\in S^{d}$, the Pucci extremal operators $P_{\lambda,\Lambda}^{\pm}:S^{d}\rightarrow \mathbb{R}$ are defined as follows
	$$P_{\lambda,\Lambda}^{+}(M):=\Lambda\sum\limits_{e_{i}>0}e_{i}+\lambda\sum\limits_{e_{i}<0}e_{i}$$
	and
	$$P_{\lambda,\Lambda}^{-}(M):=\lambda\sum\limits_{e_{i}>0}e_{i}+\Lambda\sum\limits_{e_{i}<0}e_{i},$$
	where $\{e_{i}\}_{i=1}^{d}$ are the eigenvalues of the matrix $M$. 
\end{definition}

With the Pucci extremal operators in hand, the uniformly $(\lambda,\Lambda)$-ellipticity of the operator $F$ can be reformulated as
$$P_{\lambda,\Lambda}^{-}(N)\leq F(M+N,x)-F(M,x)\leq P_{\lambda,\Lambda}^{+}(N)$$
for all $M,N\in S^{d}$.

In the sequel, we shall focus on the following equations
\begin{equation}\label{model22}
	G(D^{2}u,Du,x):=f(x)-\Phi(\abs{Du}, x)F(D^2 u, x)-H(Du,x)=0 \quad  \text{in} \quad \Omega.
\end{equation}
On account of completeness, we now give the notion of viscosity solution for the operator $G$, which was
introduced in \cite[Definition 2.7]{Birindelli2004} and \cite[Definition 2.1]{Birindelli2010JDE}.
\begin{definition}[\bf Viscosity solutions]
	\label{dingyi}
	 A function $u\in C(\Omega)$ is a viscosity supersolution (resp. subsolution) to \eqref{model22}, if for every $x_0 \in \Omega$ either there exists $\eta>0$ such that $u$ is constant in $B_{\eta}(x_{0})$ and $f(x)\geq 0$ (resp. $f(x)\leq 0$) for all $x\in B_{\eta}(x_{0})$, or, for all $\varphi \in C^2\left(\Omega\right)$ such that $ u -\varphi$ attains a local minimum (resp. local maximum) at $x_0$ and $D\varphi(x_{0})\neq 0$, it holds
	$$
	G(D^2\varphi(x_0),D\varphi(x_0),x_{0}) \geq  0 \quad({\rm resp.}\; 	G(D^2\varphi(x_0),D\varphi(x_0),x_{0})  \leq  0).
	$$
	Finally, a function $u$ is said to be a viscosity solution of \eqref{model22} if it is simultaneously a viscosity supersolution and a viscosity subsolution.
\end{definition}

We end this subsection with the following assertion from \cite[Proposition 1.2]{Birindelli2015} or \cite[Proposition 2.1]{Baasandorj}, which will be essential for handling the singular case in subsequent proofs.
\begin{proposition}\label{qiyizhuantuihua}
	Suppose that assumptions \eqref{A1}-\eqref{A5} are in force with $-1<i(\Phi)<0$ and that $u$ is a viscosity solution of \eqref{model} in the sense of Definition \ref{dingyi}. Then, $u$ is a classical viscosity solution of
	\begin{equation}\label{model1111}
		\abs{Du}^{-i(\Phi)}\Phi(\abs{Du}, x)F(D^2 u, x)+\abs{Du}^{-i(\Phi)}H(Du,x) =\abs{Du}^{-i(\Phi)}f(x) \quad  \text{in} \quad \Omega.
	\end{equation}
\end{proposition}
%%%%%%%%%%%%%%%%%%%%%%%%%%%%%%%%%%%%%%%%%%%%%%%%%%%%%%%%%%%%%%%%%%%%%%%%%%%%%%%%%%%%%%%%%%%%%%%%%%%%%%%%%%%%%%
\subsection{Smallness regime}
In this subsection, we will apply the scaling features of \eqref{model} to trace the problem back to a smallness regime. That is, without loss of generality, we explicitly verify that it is possible to suppose that
\begin{equation}\label{small}
	\|{u}\|_{L^{\infty}\left(B_{1}\right)}\leq 1 \quad {\rm and}\quad \max\left\{\|{\rm osc}_{{F}}\|_{L^{\infty}\left(B_{1}\right)},\|{f}\|_{L^{\infty}\left(B_{1}\right)},\mathcal{K},\mathcal{M}\right\}\leq \delta
\end{equation}
for some constant $0<\delta<1$, and also that $\nu_{0}=\nu_{1}=1$. In what follows, we examine its scaling properties that allow us to work under assumption $\eqref{small}$. Let $u$ be a viscosity solution to \eqref{model}. For any fixed point $x_{0}\in \Omega^{\prime}\subset\subset\Omega$, we define $\tilde{u}:B_{1}\rightarrow \mathbb{R}$ by
	$$\tilde{u}(x)=\frac{u(rx+x_{0})}{K},$$
	where $K\geq 1\geq r$ are constants to be determined later. Let us first observe that $B_{r}(x_{0})\subset\subset\Omega$, due to the selection of $r$ below. Then we can readily check that
	$\tilde{u}$ solves
	\begin{equation} \label{switchequation}
			\tilde{\Phi}(\Abs{D\tilde{u}}, x)	\tilde{F}(D^2 \tilde{u}, x)+\tilde{H}(D\tilde{u},x) = \tilde{f}(x) \quad \text{in} \quad  B_{1}
	\end{equation}
in the viscosity sense, where
	\begin{align*}
		\tilde{F}(X,x):=&\frac{r^{2}}{K}F\left(\frac{K}{r^{2}}X,rx+x_{0}\right),\\
		\tilde{\Phi}(t,x):=&\frac{\Phi\left(\frac{K}{r}t,rx+x_{0}\right)}{\Phi\left(\frac{K}{r},rx+x_{0}\right)},\\
		\tilde{H}(t,x):=&\frac{r^{2}}{K\Phi\left(\frac{K}{r},rx+x_{0}\right)}H\left(\frac{K}{r}t,rx+x_{0}\right),\\
		\tilde{f}(x):=&\frac{r^{2}}{K\Phi\left(\frac{K}{r},rx+x_{0}\right)}f(rx+x_{0}).
	\end{align*}
	Note that $\tilde{F}$ is still a uniformly $(\lambda,\Lambda)$-elliptic operator, the map $t\mapsto \frac{\tilde{\Phi}(t,x)}{t^{i(\Phi)}}$ is almost non-decreasing, and the map $t\mapsto \frac{\tilde{\Phi}(t,x)}{t^{s(\Phi)}}$ is almost non-increasing with the same constant
	$L\geq 1$ as in assumption \eqref{A3}; and $\tilde{\Phi}(1,x)=1$ for all $x\in B_{1}$. Moreover, a direct calculation yields that
	\begin{equation*}
		{\rm osc}_{\tilde{F}}(x,0)=\sup\limits_{M\in S^{d}\setminus \{0\}}\frac{\Abs{F\left(\frac{K}{r^{2}}M,rx+x_{0}\right)-F\left(\frac{K}{r^{2}}M,x_{0}\right)}}{\frac{K}{r^{2}}\|M\|}={\rm osc}_{{F}}(rx+x_{0},x_{0}).
	\end{equation*}
This along with \eqref{A2} leads to
	\begin{equation*}
	\|{\rm osc}_{\tilde{F}}\|_{L^{\infty}\left(B_{1}\right)}= \|{\rm osc}_{{F}}(\cdot,x_{0})\|_{L^{\infty}\left(B_{r}(x_{0})\right)}\leq C_{F}r^{\theta}.
\end{equation*}
Combining \eqref{A3} with \eqref{A4} and $\frac{K}{r}\geq 1$, we arrive at
\begin{equation*}
		|\tilde{H}(t,x)|\leq \frac{Lr^{2+i(\Phi)}}{\nu_{0}K^{1+i(\Phi)}}\bigg(\mathcal{K}+\mathcal{M}\left(\frac{K}{r}\right)^{m}|t|^{m}\bigg)=:\tilde{\mathcal{K}}+\tilde{\mathcal{M}}|t|^{m},
\end{equation*}
\begin{equation*}
	\|\tilde{f}\|_{L^{\infty}\left(B_{1}\right)}\leq \frac{Lr^{2+i(\Phi)}}{\nu_{0}K^{1+i(\Phi)}}\|{f}\|_{L^{\infty}\left(\Omega\right)},
\end{equation*}
%It follows from \eqref{A3} and $\frac{K}{r}\geq 1$ that
%	\begin{equation*}
%		\|\tilde{f}\|_{L^{\infty}\left(B_{1}\right)}\leq \frac{Lr^{2+i(\Phi)}}{\nu_{0}K^{1+i(\Phi)}}\|{f}\|_{L^{\infty}\left(\Omega\right)},
%	\end{equation*}
%\begin{equation*}
%	\|\tilde{h}\|_{L^{\infty}\left(B_{1}\right)}\leq \frac{Lr^{2-m+i(\Phi)}}{\nu_{0}K^{1-m+i(\Phi)}}\|{h}\|_{L^{\infty}\left(\Omega\right)}.
%\end{equation*}
Now, for given $\delta\in(0,1)$, which will be sufficiently small but fixed. Then, we select
\begin{equation*}
	K:=\left\{
	\begin{array}{lcl}
		 1+\|u\|_{L^{\infty}\left(\Omega\right)}+\left(\frac{L}{\nu_{0}}\left(\|f\|_{L^{\infty}\left(\Omega\right)}+\mathcal{K}\right)\right)^{\frac{1}{1+i(\Phi)}}+\left(\frac{L\mathcal{M}}{\nu_{0}}\right)^{\frac{1}{1+i(\Phi)-m}} & \text{for} & m<1+i(\Phi), \\
		1+\|u\|_{L^{\infty}\left(\Omega\right)} &\text{for}&  m=1+i(\Phi),
	\end{array}
	\right.
\end{equation*}
and
\begin{equation*}
r:=\left\{
\begin{array}{lcl}
	 \min\left\{1,\frac{1}{4}\text{dist}(\Omega^{\prime},\partial\Omega),\left(\frac{\delta}{C_{F}}\right)^{\frac{1}{\theta}},\delta^{\frac{1}{2+i(\Phi)}},\delta^{\frac{1}{2+i(\Phi)-m}}\right\} & \text{for} & m<1+i(\Phi), \\
	 \min\left\{1,\frac{1}{4}\text{dist}(\Omega^{\prime},\partial\Omega),\left(\frac{\delta}{C_{F}}\right)^{\frac{1}{\theta}},\left(\frac{\delta\nu_{0}}{L\left(\|f\|_{L^{\infty}(\Omega)}+\mathcal{K}\right)}\right)^{\frac{1}{2+i(\Phi)}},\frac{\delta\nu_{0}}{L\mathcal{M}}\right\} &\text{for}&  m=1+i(\Phi).
\end{array}
\right.
\end{equation*}
With such choice, we arrive at
\begin{equation*}
	\|\tilde{u}\|_{L^{\infty}\left(B_{1}\right)}\leq \frac{\|{u}\|_{L^{\infty}\left(\Omega\right)}}{K}\leq 1\quad {\rm and }\quad
	\max\left\{\|{\rm osc}_{\tilde{F}}\|_{L^{\infty}\left(B_{1}\right)},\|\tilde{f}\|_{L^{\infty}\left(B_{1}\right)},\tilde{\mathcal{K}},\tilde{\mathcal{M}}\right\}
	\leq \delta.
\end{equation*}	
Therefore, $\tilde{u}$ solves an equation possessing the same structure as \eqref{model} and $\tilde{u}$ is in the smallness regime.
\begin{remark}\label{zhu2.5}
Clearly, it is enough to prove Theorem $\ref{main}$ for $\tilde{u}\in C(B_{1})$ of \eqref{switchequation}. In fact, as soon as we show that
$$[\tilde{u}]_{C^{1, \alpha}(B_{1/2})}\leq C,$$
by scaling back to $u$, we get \\
(i) if $m<1+i(\Phi)$, then
$$[u]_{C^{1, \alpha}(B_{r}(x_{0}))}\leq
\frac{CK}{r^{\alpha}} \leq C\left(N,\lambda,\Lambda,L,i(\Phi),m,\theta,\alpha,\text{dist}(\Omega^{\prime},\partial\Omega),C_{F}\right)K;$$
(ii) if $m=1+i(\Phi)$, then
$$[u]_{C^{1, \alpha}(B_{r}(x_{0}))}\leq
\frac{CK}{r^{\alpha}} \leq C\left(N,\lambda,\Lambda,L,i(\Phi),m,\theta,\alpha,\nu_{0},\text{dist}(\Omega^{\prime},\partial\Omega),C_{F},\|f\|_{L^{\infty}(\Omega)},\mathcal{K},\mathcal{M}\right)K.$$
Finally, we can conclude that the estimates \eqref{giji1} and \eqref{giji2} as stated in Theorem \ref{model} directly follow via a standard covering argument.
\end{remark}
%%%%%%%%%%%%%%%%%%%%%%%%%%%%%%%%%%%%%%%%%%%%%%%%%%%%%%%%
%%%%%%%%%%%%%%%%%%%%%%%%%%%%%%%%%%%%%%%%%%%%%%%%%%%%%%%
\section{H\"{o}lder continuity of perturbed equations}\label{section3}
In this section, we obtain local H\"{o}lder continuity estimates of viscosity solutions to
\begin{equation}\label{model111}
	\Phi(\abs{Du+\xi}, x)F(D^2 u, x)+H(Du+\xi,x) =f(x) \quad  \text{in} \quad B_{1},
\end{equation}
where $\xi$ is an arbitrary vector in $\rn$.
Such estimates yield compactness with respect to uniform convergence to a large class of functions related to the equations we propose to study. The proof relies on the celebrated Crandall-Ishii-Lions lemma, we refer the reader to \cite[Theorem 3.2]{Crandle1} or \cite[Proposition II.3]{Lions1}. Hereafter in this paper, we say that $u\in C(B_{1})$ is a normalized viscosity solution if
$\| u\|_{L^{\infty}(B_{1})} \leq 1$.

To begin with, we deal with the scenario $0\leq i(\Phi)<m\leq 1+i(\Phi)$. In this case, we need to impose the additional assumption to restrain the growth of the term $\mathcal{M}\abs{Du+\xi}^{m-i(\Phi)}$.
\begin{proposition}\label{jin1}
	Assume that the assumptions \eqref{A1}-\eqref{A5} hold with $\nu_{0}=\nu_{1}=1$ and $0\leq i(\Phi)<m\leq 1+i(\Phi)$.
	Let $\xi\in \rn$ and $u\in C(B_{1})$ be a normalized viscosity solution of \eqref{model111}. %with $\|f\|_{L^{\infty}(B_{1})}\leq \delta<1$.
	There exists a universal constant $\kappa_{0}>0$ such that if
	\begin{equation}\label{control}
		\mathcal{M}\left(\abs{\xi}^{m-i(\Phi)}+1\right)\leq \kappa
	\end{equation}
	for some $\kappa\leq \kappa_{0}$, then $u\in C_{loc}^{0,\gamma}(B_{1})$ for some $\gamma\in(0,1)$. Furthermore, there exists a constant $C>0$ depending only on $d,\lambda,\Lambda,L,m,i(\Phi),C_{F},\theta,\kappa_{0}$, $\gamma$, $\|f\|_{L^{\infty}(B_{1})}$ and $\mathcal{K}$, such that
	\begin{equation}
		[u]_{C^{0,\gamma}(B_{1/2})}\leq C.
	\end{equation}
\end{proposition}
\begin{proof}
	Let $M_{0}:=\left(\frac{\kappa_{0}}{\mathcal{M}}\right)^{\frac{1}{m-i(\Phi)}}$. It follows from $m>i(\Phi)$ and \eqref{control} that
\begin{equation}\label{bound}
	M_{0}\geq 1\quad {\rm and} \quad \abs{\xi}\leq M_{0}.
\end{equation}
	Let $0<r<1$ be fixed. We consider the quantity
	\begin{equation*}
		 G(x_{0}):=\sup\limits_{(x,y)\in B_{r}\times B_{r}}\left\{u(x)-u(y)-L_{1}\lvert x-y \rvert^{\gamma}-L_{2}\Big(\lvert x-x_{0}\rvert^{2}+\lvert y-x_{0} \rvert^{2} \Big)\right\}
	\end{equation*}
defined for each $x_{0} \in B_{r/2}$, where $\gamma\in(0,1)$ and $L_{1},L_{2}>1$. If we prove that there exist two constants $L_{1}, L_{2} >1$ such that
\begin{equation}\label{goal}
	G(x_{0})\leq 0
\end{equation}
for all $x_{0} \in B_{r/2}$, then the result is established. As it is usual when resorting to this class of arguments, we reason through a contradiction argument. That is to say, suppose that for all $L_{1},L_{2}>1$, there is $x_{0}\in B_{r/2}$ for which $G(x_{0})>0$. For ease of the presentation, we introduce two auxiliary functions $\psi,\Psi:\overline{B}_{r}\times\overline{B}_{r}\rightarrow \mathbb{R}$, to be defined as
\begin{equation*}
		\begin{cases}
				\psi(x,y):=L_{1}\lvert x-y \rvert^{\gamma}+L_{2}\Big(\lvert x-x_{0}\rvert^{2}+\lvert y-x_{0} \rvert^{2} \Big),\\
					\Psi\left(x,y\right):=u(x)-u(y)-\psi(x,y).
		\end{cases}
	\end{equation*}
We denote by $\left(\hat{x},\hat{y}\right)$ a maximum point of $\Psi(x,y)$ in $\overline{B}_{r}\times\overline{B}_{r}$, namely,
	\begin{equation*}
		\Psi\left(\hat{x},\hat{y}\right)= G(x_{0})>0.
	\end{equation*}
This together with $\|u\|_{L^{\infty}(B_{1})}\leq 1$ yields
\begin{equation}\label{xiao}
	L_{1}\lvert \hat{x}-\hat{y} \rvert^{\gamma}+L_{2}\Big(\lvert \hat{x}-x_{0}\rvert^{2}+\lvert \hat{y}-x_{0} \rvert^{2} \Big)<u(\hat{x})-u(\hat{y})\leq 2\|u\|_{L^{\infty}(B_{1})}\leq 2.
\end{equation}
It follows from \eqref{xiao} that
\begin{equation}\label{youyong111}
	\abs{\hat{x}-\hat{y}}\leq \sqrt{2\left(\lvert \hat{x}-x_{0}\rvert^{2}+\lvert \hat{y}-x_{0} \rvert^{2} \right)}\leq \frac{2}{\sqrt{L_{2}}}.
\end{equation}	
	Before we proceed, choosing $L_{2}\geq \frac{32}{r^{2}}$. With this choice, we deduce that
\begin{equation}\label{neidian11}
	\lvert \hat{x}-x_{0}\rvert,\lvert \hat{y}-x_{0} \rvert\leq \frac{r}{4},\quad \abs{\hat{x}-\hat{y}}\leq \lvert \hat{x}-x_{0}\rvert+\lvert \hat{y}-x_{0} \rvert\leq \frac{r}{2}<1.
\end{equation}
A combination of \eqref{neidian11} and $x_{0}\in B_{r/2}$ yields that $\left(\hat{x},\hat{y}\right)\in B_{r}\times B_{r}$. In addition, notice that $\hat{x}\neq\hat{y}$; otherwise $G(x_{0})=\Psi\left(\hat{x},\hat{y}\right)=-2L_{2}\abs{\hat{x}-x_{0}}^{2}\leq 0$ and \eqref{goal} would be immediately verified.
%By virtue of \eqref{neidian11} and the triangle inequality, we obtain
%	$$\abs{\hat{x}-\hat{y}}\leq \lvert \hat{x}-x_{0}\rvert+\lvert \hat{y}-x_{0} \rvert\leq \frac{r}{2}<1.$$
%Moreover, it follows from \eqref{xiao} that
%\begin{equation}\label{youyong111}
%	\abs{\hat{x}-\hat{y}}\leq \sqrt{2\left(\lvert \hat{x}-x_{0}\rvert^{2}+\lvert \hat{y}-x_{0} \rvert^{2} \right)}\leq \frac{2}{\sqrt{L_{2}}}.
%\end{equation}	

	 We are in a position to apply the Crandall-Ishii-Lions lemma (see \cite[Theorem 3.2]{Crandle1}, \cite[Proposition 2.1]{Fili}) to assure
	the existence of a limiting subjet $\left(\xi_{\hat{x}},X\right)$ of $u$ at $\hat{x}$ and a limiting superjet $\left(\xi_{\hat{y}},Y\right)$ of $u$ at $\hat{y}$, such that the matrices $X,Y\in S^{d}$ satisfy the matrix inequality
	\begin{equation}\label{matrix}
		\left(
		\begin{array}{cc}
			X & 0 \\
			0 & -Y \\
		\end{array}
		\right)
		\leq  \left(
		\begin{array}{cc}
			A & -A \\
			-A & A \\
		\end{array}
		\right)+
		(2L_{2}+\epsilon)
		\left(
		\begin{array}{cc}
			I & 0 \\
			0 & I \\
		\end{array}
		\right)
	\end{equation}
	with $\epsilon\in(0,1)$, that only depends on the norm of $A$ and can be made sufficiently small. Here,
	\begin{equation*}
			\xi_{\hat{x}}:=\gamma L_{1}(\hat{x}-\hat{y})\lvert \hat{x}-\hat{y}\rvert^{\gamma -2}+2L_{2}(\hat{x}-x_{0}),  \quad \xi_{\hat{y}}:=\gamma L_{1}(\hat{x}-\hat{y})\lvert \hat{x}-\hat{y}\rvert^{\gamma -2}-2L_{2}(\hat{y}-x_{0}),
	\end{equation*}
\begin{equation}\label{Adingyi}
	A:=L_{1}\gamma\left[(\gamma-2)\abs{\hat{x}-\hat{y}}^{\gamma-4}\left((\hat{x}-\hat{y})\otimes(\hat{x}-\hat{y})\right)+\abs{\hat{x}-\hat{y}}^{\gamma-2}I\right].
\end{equation}
Furthermore, we have the following viscosity inequalities
\begin{equation}\label{jie1}
\Phi(\abs{\xi_{\hat{x}}+\xi},\hat{x})F(X, \hat{x})+H(\xi_{\hat{x}}+\xi,\hat{x})\geq f(\hat{x}),
\end{equation}
\begin{equation}\label{jie2}
	\Phi(\abs{\xi_{\hat{y}}+\xi},\hat{y})F(Y, \hat{y})+H(\xi_{\hat{y}}+\xi,\hat{y})\leq f(\hat{y}).
\end{equation}
In what follows, for ease of clarity, we split the  proof into three steps.\\
{\bf Step 1.} We first claim that
\begin{equation}\label{shangxiajie}
	\frac{\gamma L_{1}}{2}\Abs{\hat{x}-\hat{y}}^{\gamma-1}\leq\abs{\xi_{\hat{x}}},\abs{\xi_{\hat{y}}}\leq 2\gamma L_{1}\Abs{\hat{x}-\hat{y}}^{\gamma-1}.
\end{equation}	
In fact, choose $L_{1}>\frac{L_{2}r^{2-\gamma}}{\gamma2^{1-\gamma}}$, it follows from \eqref{neidian11} %$\Abs{\hat{x}-\hat{y}}\leq\frac{r}{2}$, 
and $\gamma<1$ that
$$2L_{2}\Abs{\hat{x}-x_{0}}\leq 2L_{2}\frac{r}{4}< \frac{\gamma L_{1}}{2}\left(\frac{r}{2}\right)^{\gamma-1}\leq \frac{\gamma L_{1}}{2}\Abs{\hat{x}-\hat{y}}^{\gamma-1}.
$$	
This along with the triangle inequality leads to that
\begin{equation*}
	\abs{\xi_{\hat{x}}}\geq \gamma L_{1}\lvert \hat{x}-\hat{y}\rvert^{\gamma -1}-2L_{2}\Abs{\hat{x}-x_{0}}
	\geq \gamma L_{1}\lvert \hat{x}-\hat{y}\rvert^{\gamma -1}-\frac{\gamma L_{1}}{2}\Abs{\hat{x}-\hat{y}}^{\gamma-1}= \frac{\gamma L_{1}}{2}\Abs{\hat{x}-\hat{y}}^{\gamma-1},
\end{equation*}
\begin{equation*}
	\abs{\xi_{\hat{x}}}\leq \gamma L_{1}\lvert \hat{x}-\hat{y}\rvert^{\gamma -1}+2L_{2}\Abs{\hat{x}-x_{0}}
	\leq \gamma L_{1}\lvert \hat{x}-\hat{y}\rvert^{\gamma -1}+\frac{\gamma L_{1}}{2}\Abs{\hat{x}-\hat{y}}^{\gamma-1}\leq 2\gamma L_{1}\Abs{\hat{x}-\hat{y}}^{\gamma-1}.
\end{equation*}
In exactly the same way, we derive
\begin{equation*}
	\frac{\gamma L_{1}}{2}\Abs{\hat{x}-\hat{y}}^{\gamma-1}\leq \abs{\xi_{\hat{y}}}\leq 2\gamma L_{1}\Abs{\hat{x}-\hat{y}}^{\gamma-1}.
\end{equation*}	
We proceed to estimate the bounds of $\abs{\xi_{\hat{x}}+\xi}$ and $\abs{\xi_{\hat{y}}+\xi}$.
Selecting $L_{1}>\frac{4M_{0}}{\gamma}$, then a combination of \eqref{bound} with \eqref{shangxiajie} and $\Abs{\hat{x}-\hat{y}}^{\gamma-1}>1$ yields that
\begin{align}
	& \begin{aligned}
		|\xi_{\hat{x}} + \xi| &\leq |\xi_{\hat{x}}| + |\xi| \leq 2\gamma L_1 |\hat{x}-\hat{y}|^{\gamma-1} + M_0 \leq 2\gamma L_1 |\hat{x}-\hat{y}|^{\gamma-1} + \frac{\gamma L_1}{4} \\
		&\leq 2\gamma L_1 |\hat{x}-\hat{y}|^{\gamma-1} + \gamma L_1 |\hat{x}-\hat{y}|^{\gamma-1} = 3\gamma L_1 |\hat{x}-\hat{y}|^{\gamma-1},
	\end{aligned} \label{shangbound} \\
	& \begin{aligned}
		|\xi_{\hat{x}} + \xi| &\geq |\xi_{\hat{x}}| - |\xi| \geq \frac{\gamma L_1}{2} |\hat{x}-\hat{y}|^{\gamma-1} - M_0 \geq \frac{\gamma L_1}{2} |\hat{x}-\hat{y}|^{\gamma-1} - \frac{\gamma L_1}{4} |\hat{x}-\hat{y}|^{\gamma-1} \\
		&= \frac{\gamma L_1}{4} |\hat{x}-\hat{y}|^{\gamma-1} > M_0 \, |\hat{x}-\hat{y}|^{\gamma-1} > 1.
	\end{aligned} \label{xiabound}
\end{align}
By the same way, we can deduce
\begin{equation}\label{ayj}
1<M_{0}\Abs{\hat{x}-\hat{y}}^{\gamma-1}\leq \abs{\xi_{\hat{y}}+\xi} \leq 3\gamma L_{1}\Abs{\hat{x}-\hat{y}}^{\gamma-1}.
\end{equation}
As a consequence, in light of condition \eqref{A3}, \eqref{xiabound} and \eqref{ayj}, we can rewrite \eqref{jie1} and \eqref{jie2} as
\begin{equation*}
	F(X, \hat{x})+\frac{H(\xi_{\hat{x}}+\xi,\hat{x})}{\Phi(\abs{\xi_{\hat{x}}+\xi},\hat{x})}\geq \frac{f(\hat{x})}{\Phi(\abs{\xi_{\hat{x}}+\xi},\hat{x})},
\end{equation*}
\begin{equation*}
	F(Y, \hat{y})+\frac{H(\xi_{\hat{y}}+\xi,\hat{y})}{\Phi(\abs{\xi_{\hat{y}}+\xi},\hat{y})}\leq \frac{f(\hat{y})}{\Phi(\abs{\xi_{\hat{y}}+\xi},\hat{y})}.
\end{equation*}
Combining the uniform ellipticity of operator $F$ at the point $\hat{x}$ with \eqref{A2}, we deduce
\begin{equation*}
	\begin{split}
		F(X,\hat{x})\leq& P_{\lambda,\Lambda}^{+}(X-Y)+F(Y,\hat{y})+F(Y,\hat{x})-F(Y,\hat{y})\\
		\leq& P_{\lambda,\Lambda}^{+}(X-Y)+F(Y,\hat{y})+C_{F}\|Y\|\abs{\hat{x}-\hat{y}}^{\theta}.
	\end{split}
\end{equation*}
Combining the last three displays, we arrive at
\begin{equation}\label{ainequality}
	\begin{split}
		&\frac{f(\hat{x})}{\Phi(\abs{\xi_{\hat{x}}+\xi},\hat{x})}
		-\frac{H(\xi_{\hat{x}}+\xi,\hat{x})}{\Phi(\abs{\xi_{\hat{x}}+\xi},\hat{x})}\\
		\leq& P_{\lambda,\Lambda}^{+}(X-Y)+\frac{f(\hat{y})}{\Phi(\abs{\xi_{\hat{y}}+\xi},\hat{y})}-\frac{H(\xi_{\hat{y}}+\xi,\hat{y})}{\Phi(\abs{\xi_{\hat{y}}+\xi},\hat{y})} +C_{F}\|Y\|\abs{\hat{x}-\hat{y}}^{\theta}.
	\end{split}
\end{equation}
{\bf Step 2.} We now estimate $P_{\lambda,\Lambda}^{+}(X-Y)$ and $\|Y\|$. For vectors of the form $(z,z)\in\mathbb{R}^{2d}$ with $\abs{z}=1$, we apply the matrix inequality \eqref{matrix} to obtain
	\begin{equation*}
		\left\langle(X-Y)z,z\right\rangle\leq \left(4L_{2}+2\epsilon\right)|z|^{2}.
	\end{equation*}
	This means that all the eigenvalues of $X-Y$ are less than or equal to $4L_{2}+2\epsilon$. In particular, applying \eqref{matrix} to the vector $\left(\frac{\hat{x}-\hat{y}}{\abs{\hat{x}-\hat{y}}},\frac{\hat{y}-\hat{x}}{\abs{\hat{x}-\hat{y}}}\right)\in\mathbb{R}^{2d}$,
	we get
	\begin{equation*}
		\begin{split}
			\left\langle(X-Y)\frac{\hat{x}-\hat{y}}{\abs{\hat{x}-\hat{y}}},\frac{\hat{x}-\hat{y}}{\abs{\hat{x}-\hat{y}}}\right\rangle&\leq 4\left\langle A\frac{\hat{x}-\hat{y}}{\abs{\hat{x}-\hat{y}}},\frac{\hat{x}-\hat{y}}{\abs{\hat{x}-\hat{y}}}\right\rangle+(4L_{2}+2\epsilon)\Abs{\frac{\hat{x}-\hat{y}}{\abs{\hat{x}-\hat{y}}}}^{2}\\
			 &=\left(4L_{1}\gamma(\gamma-1)\abs{\hat{x}-\hat{y}}^{\gamma-2}+4L_{2}+2\epsilon\right)\Abs{\frac{\hat{x}-\hat{y}}{\abs{\hat{x}-\hat{y}}}}^{2}.
		\end{split}
	\end{equation*}	
	This yields that
	at least one eigenvalue of $X-Y$ is less than $4L_{1}\gamma(\gamma-1)\abs{\hat{x}-\hat{y}}^{\gamma-2}+4L_{2}+2\epsilon$. We choose $L_{1}>\frac{4L_{2}+2}{4\gamma(1-\gamma)}$, then it leads to
	$$4L_{1}\gamma(\gamma-1)\abs{\hat{x}-\hat{y}}^{\gamma-2}+4L_{2}+2\epsilon<4L_{1}\gamma(\gamma-1)+4L_{2}+2<0.$$
	This means that at least one eigenvalue of
	$X-Y$ is negative. By the definition of Pucci extremal operator, we deduce
	\begin{equation}\label{suanzijie}
		P_{\lambda,\Lambda}^{+}(X-Y)\leq \Lambda(d-1)(4L_{2}+2\epsilon)+\lambda\left(4L_{2}+2\epsilon-4\gamma(1-\gamma)L_{1}\abs{\hat{x}-\hat{y}}^{\gamma-2}\right).
	\end{equation}
	Further, applying \eqref{matrix} to the vectors $(0,z)\in\mathbb{R}^{2d}$ with $\abs{z}=1$,
	we obtain
	\begin{equation}\label{Yfanshu}
		\left\langle-Yz,z\right\rangle\leq \left\langle Az,z\right\rangle+(2L_{2}+\epsilon)|z|^{2}.
	\end{equation}
It follows from the definition of the matrix $A$ in \eqref{Adingyi} that $A- L_{1}\gamma\abs{\hat{x}-\hat{y}}^{\gamma-2}I$ is non-positive definite matrix. This together with \eqref{Yfanshu} yields that
	\begin{equation}\label{fanshujie}
		\|Y\|\leq L_{1}\gamma\abs{\hat{x}-\hat{y}}^{\gamma-2}+2L_{2}+\epsilon.
	\end{equation}
Combining \eqref{fanshujie} with $\abs{\hat{x}-\hat{y}}^{\theta}<1$, we have
\begin{equation}\label{afanshujie}
	C_{F}\|Y\|\Abs{\hat{x}-\hat{y}}^{\theta}\leq C_{F}L_{1}\gamma\Abs{\hat{x}-\hat{y}}^{\theta+\gamma-2}+C_{F}\left(2L_{2}+\epsilon\right).
\end{equation}
{\bf Step 3.} We proceed to estimate the remaining four terms in \eqref{ainequality}. A combination of \eqref{shangbound}-\eqref{ayj} with the assumptions \eqref{A3}, \eqref{A4} and $m>i(\Phi)\geq 0$ yields that
\begin{equation*}
	\frac{f(\hat{y})}{\Phi\left(\abs{\xi_{\hat{y}}+\xi},\hat{y}\right)}\leq  \frac{L\|f\|_{L^{\infty}(B_{1})}}{\abs{\xi_{\hat{y}}+\xi}^{i(\Phi)}}\leq L\|f\|_{L^{\infty}(B_{1})},
\end{equation*}
\begin{equation*}
	\frac{f(\hat{x})}{\Phi\left(\abs{\xi_{\hat{x}}+\xi},\hat{x}\right)}\geq  \frac{-L\|f\|_{L^{\infty}(B_{1})}}{\abs{\xi_{\hat{x}}+\xi}^{i(\Phi)}}\geq -L\|f\|_{L^{\infty}(B_{1})},
\end{equation*}
\begin{equation*}
	\frac{H(\xi_{\hat{x}}+\xi,\hat{x})}{\Phi\left(\abs{\xi_{\hat{x}}+\xi},\hat{x}\right)}\leq \frac{L\left(\mathcal{K}+\mathcal{M}\abs{\xi_{\hat{x}}+\xi}^{m}\right)}{\abs{\xi_{\hat{x}}+\xi}^{i(\Phi)}}\leq L\mathcal{K}+
	L\mathcal{M}\left(3\gamma L_{1}\Abs{\hat{x}-\hat{y}}^{\gamma-1}\right)^{m-i(\Phi)},
\end{equation*}
\begin{equation*}
	\frac{H(\xi_{\hat{y}}+\xi,\hat{y})}{\Phi\left(\abs{\xi_{\hat{y}}+\xi},\hat{y}\right)}\leq \frac{L\left(\mathcal{K}+\mathcal{M}\abs{\xi_{\hat{y}}+\xi}^{m}\right)}{\abs{\xi_{\hat{y}}+\xi}^{i(\Phi)}}\leq
	L\mathcal{K}+
	L\mathcal{M}\left(3\gamma L_{1}\Abs{\hat{x}-\hat{y}}^{\gamma-1}\right)^{m-i(\Phi)}.
\end{equation*}
Substituting the aforementioned four inequalities, \eqref{suanzijie} and \eqref{afanshujie} into \eqref{ainequality}, in combination with $\Abs{\hat{x}-\hat{y}}<1$, $\gamma<1$ and $m\leq 1+i(\Phi)$, we eventually arrive at
\begin{equation*}
	\begin{split}
		&-2L\|f\|_{L^{\infty}(B_{1})}-2L\mathcal{K}\\
		\leq& C_{1}+L_{1}\Abs{\hat{x}-\hat{y}}^{\gamma-2}\left(C_{F}\gamma\Abs{\hat{x}-\hat{y}}^{\theta}-4\lambda\gamma(1-\gamma)+6L\mathcal{M}\Abs{\hat{x}-\hat{y}}^
		{(\gamma-1)\left(m-i(\Phi)-1\right)+1}\right)\\
		\leq& C_{1}+L_{1}\abs{\hat{x}-\hat{y}}^{\gamma-2}\left(C_{F}\gamma\abs{\hat{x}-\hat{y}}^{\theta}-4\lambda\gamma(1-\gamma)+6L\mathcal{M}\Abs{\hat{x}-\hat{y}}\right),
	\end{split}
\end{equation*}
where $C_{1}:=\left(\Lambda(d-1)+\lambda\right)(4L_{2}+2)+C_{F}(2L_{2}+1)$. Choose $$L_{2}\geq \max\left\{4\left(\frac{C_{F}}{\lambda(1-\gamma)}\right)^{2/\theta},\left(\frac{12L\mathcal{M}}{\lambda\gamma(1-\gamma)}\right)^{2}\right\}.$$ With this choice, we apply \eqref{youyong111} to derive
\begin{equation}\label{zuihou}
	C_{F}\Abs{\hat{x}-\hat{y}}^{\theta}\leq C_{F}\left(\frac{2}{\sqrt{L_{2}}}\right)^{\theta}\leq \lambda(1-\gamma),\quad 6L\mathcal{M}\Abs{\hat{x}-\hat{y}}\leq \frac{12L\mathcal{M}}{\sqrt{L_{2}}}\leq \lambda\gamma(1-\gamma).
\end{equation}
Utilizing \eqref{zuihou}, $\Abs{\hat{x}-\hat{y}}<1$ and $\gamma<1$, we further deduce
\begin{equation}\label{key11}
	-2L\left(\|f\|_{L^{\infty}(B_{1})}+\mathcal{K}\right)
	\leq C_{1}-2L_{1}\lambda\gamma(1-\gamma)\abs{\hat{x}-\hat{y}}^{\gamma-2}< C_{1}-2L_{1}\lambda\gamma(1-\gamma).
\end{equation}
Finally, if we choose $L_{1}\geq \frac{C_{1}+2L\left(\|f\|_{L^{\infty}(B_{1})}+\mathcal{K}\right)}{2\lambda\gamma(1-\gamma)}$, we obtain a contradiction with \eqref{key11}.
	
Therefore, we prove the claim \eqref{goal}, which means that $u$ is $\gamma$-H\"{o}lder continuous with the estimate
$$[u]_{C^{0,\gamma}(B_{r/2})}\leq C\left(d,\lambda,\Lambda,\kappa_{0},r,L,m,i(\Phi),C_{F},\theta,\gamma,\|f\|_{L^{\infty}(B_{1})},\mathcal{K}\right).$$
The proof is complete.	
\end{proof}
We are now ready to consider the range $0<m\leq i(\Phi)$ for the degenerate case and the range $0<m\leq 1+ i(\Phi)$ for the singular case.
\begin{proposition}\label{jin2}
	Assume that the assumptions \eqref{A1}-\eqref{A5} hold with $\nu_{0}=\nu_{1}=1$.
	Let $\xi \in \rn$ be an arbitrarily vector and $u\in C(B_{1})$ is a normalized viscosity solution of \eqref{model111}. %with $\|f\|_{L^{\infty}(B_{1})}\leq \delta<1$ and $\|h\|_{L^{\infty}(B_{1})}\leq \delta<1$.
	Then there exists a constant $A_{0}=A_{0}(d,\lambda,\Lambda,L,m,i(\Phi),C_{F},\theta,\|f\|_{L^{\infty}(B_{1})},\mathcal{K},\mathcal{M})>0$ such that
	\begin{itemize}
		\item   [{\rm$({{\rm i}})$}] if $0<m\leq i(\Phi)$ and $\abs{\xi}\geq A_{0}$, then $u\in C_{loc}^{0,1}(B_{1})$. In addition, there exists a constant $C=C(d,\lambda,\Lambda,L,m,i(\Phi),C_{F},\theta,\|f\|_{L^{\infty}(B_{1})},\mathcal{K},\mathcal{M})>0$ such that
		\begin{equation}
			[u]_{C^{0,1}(B_{1/2})}\leq C.
		\end{equation}
		\item [{\rm$({{\rm ii}})$}] If $0<m\leq i(\Phi)$ and $\abs{\xi}< A_{0}$, then $u\in C_{loc}^{0,\gamma}(B_{1})$ for some $\gamma>0$. In addition, there exists a constant $C=C(d,\lambda,\Lambda,L,m,i(\Phi),C_{F},\theta,\gamma,\|f\|_{L^{\infty}(B_{1})},\mathcal{K},\mathcal{M})>0$, such that
		\begin{equation}
			[u]_{C^{0,\gamma}(B_{1/2})}\leq C.
		\end{equation}
		\item [ {\rm$({{\rm iii}})$}] If $-1<i(\Phi)<0$, $0<m\leq 1+i(\Phi)$ and $\xi=0$, then $u\in C_{loc}^{0,1}(B_{1})$. In addition, there exists a constant $C=C(d,\lambda,\Lambda,L,m,i(\Phi),C_{F},\theta,\|f\|_{L^{\infty}(B_{1})},\mathcal{K},\mathcal{M})>0$ such that
		\begin{equation}
			[u]_{C^{0,1}(B_{1/2})}\leq C.
		\end{equation}
	\end{itemize}
\end{proposition}
\begin{proof}
	Since the proof is similar to Proposition \ref{jin1}, here we mainly concentrate on the differences. Let $0<r<1$ be fixed. For the proof of (i) and (iii), it suffices to show that there exist constants $L_{1},L_{2}>1$ such that
	\begin{equation*}%\label{goal22}
		G(x_{0}):=\sup\limits_{(x,y)\in B_{r}\times B_{r}}\left\{u(x)-u(y)-L_{1}\omega(\abs{x-y})-L_{2}\Big(\lvert x-x_{0}\rvert^{2}+\lvert y-x_{0} \rvert^{2} \Big)\right\}\leq 0
	\end{equation*}
for each $x_{0}\in B_{r/2}$, where
	$$
	\omega(s)=\left\{
	\begin{array}{lcl}
		s- \omega_{0}s^{1+\beta}& \text{if} & 0\leq s\leq s_{0}:=\left(\frac{1}{(1+\beta)\omega_{0}}\right)^{1/\beta}, \\
		\omega(s_{0}) &\text{if}&  s>s_{0},
	\end{array}
	\right.
	$$
	with $\beta\in(0,\theta)$. Here we choose $\omega_{0}\in\left(0,\frac{1}{1+\beta}\right)$ such that $s_{0}\geq 1$. Observe that
	$$\omega(s)\geq 0,\quad 0\leq\omega^{\prime}(s)\leq 1,\quad \omega^{\prime\prime}(s)\leq 0,\quad \forall s\geq 0.$$
	
	We argue by contradiction by assuming that there exists $x_{0}\in B_{r/2}$ so that $G(x_{0})>0$ for all $L_{1},L_{2}>1$. Now we define two auxiliary functions $\psi,\Psi:\overline{B}_{r}\times\overline{B}_{r}\rightarrow \mathbb{R}$ by
	\begin{equation*}
		\begin{cases}
			\psi(x,y):=L_{1}\omega(\abs{x-y})+L_{2}\Big(\lvert x-x_{0}\rvert^{2}+\lvert y-x_{0} \rvert^{2} \Big),\\
			\Psi\left(x,y\right):=u(x)-u(y)-\psi(x,y).
		\end{cases}
	\end{equation*}
	Let $\left(\hat{x},\hat{y}\right)\in \overline{B}_{r}\times\overline{B}_{r}$ be a maximum point for $\Psi(x,y)$. Then it follows
	\begin{equation*}
		\Psi\left(\hat{x},\hat{y}\right)= G(x_{0})>0,
	\end{equation*}
\begin{equation}\label{31}
	L_{1}\omega(\abs{x-y})+L_{2}\Big(\lvert \hat{x}-x_{0}\rvert^{2}+\lvert \hat{y}-x_{0} \rvert^{2} \Big)<u(\hat{x})-u(\hat{y})\leq 2\|u\|_{L^{\infty}(B_{1})}\leq 2.
\end{equation}
By choosing $L_{2}\geq \frac{32}{r^{2}}$, we arrive at
\begin{equation}\label{youyong}
	\lvert \hat{x}-x_{0}\rvert,\lvert \hat{y}-x_{0} \rvert\leq \frac{r}{4},\quad \abs{\hat{x}-\hat{y}}\leq \sqrt{2\left(\lvert \hat{x}-x_{0}\rvert^{2}+\lvert \hat{y}-x_{0} \rvert^{2} \right)}\leq \frac{2}{\sqrt{L_{2}}}<1.
\end{equation}
This along with $x_{0}\in B_{r/2}$ implies that $\hat{x},\hat{y}$ belongs to the interior of $B_{r}$. In addition, $\hat{x}\neq\hat{y}$; otherwise $G(x_{0})\leq 0$ trivially.

 %It follows from \eqref{31} that
%\begin{equation}\label{youyong}
%	\abs{\hat{x}-\hat{y}}\leq \sqrt{2\left(\lvert \hat{x}-x_{0}\rvert^{2}+\lvert \hat{y}-x_{0} \rvert^{2} \right)}\leq\sqrt{2\frac{2}{L_{2}}}= \frac{2}{\sqrt{L_{2}}}.
%\end{equation}		

As in the proof of Proposition \ref{jin1}, we obtain a limiting subjet $\left(\xi_{\hat{x}},X\right)$ of $u$ at $\hat{x}$ and a limiting superjet $\left(\xi_{\hat{y}},Y\right)$ of $u$ at $\hat{y}$, such that the matrices $X,Y\in S^{d}$ satisfy the matrix inequality \eqref{matrix}.
%\begin{equation}\label{matrix2}
%	\left(
%	\begin{array}{cc}
%		X & 0 \\
%		0 & -Y \\
%	\end{array}
%	\right)
%	\leq  \left(
%	\begin{array}{cc}
%		A & -A \\
%		-A & A \\
%	\end{array}
%	\right)+
%	(2L_{2}+\epsilon)
%	\left(
%	\begin{array}{cc}
%		I & 0 \\
%		0 & I \\
%	\end{array}
%	\right)
%\end{equation}
%with $\epsilon\in (0,1)$, that only depends on the norm of $A$ and can be made sufficiently small. 
Here,
\begin{equation*}
	\xi_{\hat{x}}:=L_{1}\omega^{\prime}(\abs{\hat{x}-\hat{y}})\frac{\hat{x}-\hat{y}}{\abs{\hat{x}-\hat{y}}}+2L_{2}(\hat{x}-x_{0}), \quad  \xi_{\hat{y}}:=L_{1}\omega^{\prime}(\abs{\hat{x}-\hat{y}})\frac{\hat{x}-\hat{y}}{\abs{\hat{x}-\hat{y}}}-2L_{2}(\hat{y}-x_{0}),
\end{equation*}
$$A:=L_{1}\left[\frac{\omega^{\prime}(\abs{\hat{x}-\hat{y}})}{\abs{\hat{x}-\hat{y}}}I+\left(\omega^{\prime\prime}(\abs{\hat{x}-\hat{y}})-\frac{\omega^{\prime}(\abs{\hat{x}-\hat{y}})}{\abs{\hat{x}-\hat{y}}}\right)\frac{(\hat{x}-\hat{y})\otimes(\hat{x}-\hat{y})}{\abs{\hat{x}-\hat{y}}^{2}}\right].$$
We apply \eqref{matrix} to vectors of the form $(z,z)\in\mathbb{R}^{2d}$ with $\abs{z}=1$, to obtain
\begin{equation}\label{superasati}
	\left\langle(X-Y)z,z\right\rangle\leq \left(4L_{2}+2\epsilon\right)|z|^{2}.
\end{equation}
This means that all the eigenvalues of $X-Y$ are less than or equal to $4L_{2}+2\epsilon$. In particular, applying \eqref{matrix} to the vector $\left(\frac{\hat{x}-\hat{y}}{\abs{\hat{x}-\hat{y}}},\frac{\hat{y}-\hat{x}}{\abs{\hat{x}-\hat{y}}}\right)\in\mathbb{R}^{2d}$,
we get
\begin{equation*}
		\left\langle(X-Y)\frac{\hat{x}-\hat{y}}{\abs{\hat{x}-\hat{y}}},\frac{\hat{x}-\hat{y}}{\abs{\hat{x}-\hat{y}}}\right\rangle\leq \left(4L_{1}\omega^{\prime\prime}(\abs{\hat{x}-\hat{y}})+4L_{2}+2\epsilon\right)\Abs{\frac{\hat{x}-\hat{y}}{\abs{\hat{x}-\hat{y}}}}^{2}.
\end{equation*}	
This yields that
at least one eigenvalue of $X-Y$ is less than $4L_{1}\omega^{\prime\prime}(\abs{\hat{x}-\hat{y}})+4L_{2}+2\epsilon$. We choose $L_{1}>\frac{4L_{2}+2}{4\beta(1+\beta)\omega_{0}}$, it follows from $\abs{\hat{x}-\hat{y}}<1$ and $\beta<\theta<1$ that
\begin{align*}
	 4L_{1}\omega^{\prime\prime}(\abs{\hat{x}-\hat{y}})+4L_{2}+2\epsilon&=-4L_{1}\beta(\beta+1)\omega_{0}\abs{\hat{x}-\hat{y}}^{\beta-1}+4L_{2}+2\epsilon\\
	&<-4L_{1}\beta(\beta+1)\omega_{0}+4L_{2}+2<0.
\end{align*}
This means that at least one eigenvalue of
$X-Y$ is negative. By the definition of Pucci extremal operator, we have
\begin{equation}\label{ps}
	P_{\lambda,\Lambda}^{+}(X-Y)\leq \left(\lambda+\Lambda(d-1)\right)(4L_{2}+2\epsilon)-4L_{1}\lambda\beta(\beta+1)\omega_{0}\abs{\hat{x}-\hat{y}}^{\beta-1}.
\end{equation}
Further, applying \eqref{matrix} to the vectors $(0,z)\in\mathbb{R}^{2d}$ with $\abs{z}=1$,
we get
\begin{equation*}
	\left\langle-Yz,z\right\rangle\leq \left\langle Az,z\right\rangle+2L_{2}+\epsilon.
\end{equation*}
This together with the facts that $A- L_{1}\frac{\omega^{\prime}(\abs{\hat{x}-\hat{y}})}{\abs{\hat{x}-\hat{y}}}I$ is non-positive definite matrix and $0\leq \omega^{\prime}(s)\leq 1$ for $s\geq 0$ yields that
\begin{equation*}
	\|Y\|\leq L_{1}\abs{\hat{x}-\hat{y}}^{-1}+2L_{2}+\epsilon.
\end{equation*}
Then it follows that
\begin{equation}\label{xishu}
	C_{F}\|Y\|\Abs{\hat{x}-\hat{y}}^{\theta}\leq  C_{F}L_{1}\Abs{\hat{x}-\hat{y}}^{\theta-1}+C_{F}\left(2L_{2}+\epsilon\right).
\end{equation}

{\bf Case 1.} Suppose $0<m\leq i(\Phi)$ and $\abs{\xi}\geq A_{0}$, where $A_{0}$ will be determined later. We first choose $L_{1}>L_{2}$, then it follows from $\Abs{\hat{x}-x_{0}}\leq \frac{r}{4}$ and $0\leq \omega^{\prime}(s)\leq 1$ that
\begin{equation*}
	\abs{\xi_{\hat{x}}}\leq L_{1}|\omega^{\prime}(\abs{\hat{x}-\hat{y}})|+2L_{2}\Abs{\hat{x}-x_{0}}< 2L_{1}.
\end{equation*}
Now we take $A_{0}=3L_{1}$ so that
\begin{align}\label{xjie}
	\abs{\xi_{\hat{x}}+\xi}\geq\abs{\xi}- \abs{\xi_{\hat{x}}}\geq A_{0}-2L_{1}=L_{1}>1.
\end{align}
In exactly the same way, we get
\begin{equation}\label{yjie}
	\abs{\xi_{\hat{y}}+\xi}> 1.
\end{equation}
As the same derivation of \eqref{ainequality}, we have
\begin{equation}\label{bds}
	\begin{split}
		&\frac{f(\hat{x})}{\Phi(\abs{\xi_{\hat{x}}+\xi},\hat{x})}-\frac{f(\hat{y})}{\Phi(\abs{\xi_{\hat{y}}+\xi},\hat{y})}\\
		\leq& P_{\lambda,\Lambda}^{+}(X-Y)+\frac{H(\xi_{\hat{x}}+\xi,\hat{x})}{\Phi(\abs{\xi_{\hat{x}}+\xi},\hat{x})}-\frac{H(\xi_{\hat{y}}+\xi,\hat{y})}{\Phi(\abs{\xi_{\hat{y}}+\xi},\hat{y})} +C_{F}\|Y\|\abs{\hat{x}-\hat{y}}^{\theta}.
	\end{split}
\end{equation}
A combination of \eqref{xjie}, \eqref{yjie}, the assumptions \eqref{A3}, \eqref{A4}, and $0<m\leq i(\Phi)$ yields that
\begin{equation}\label{1}
	\frac{f(\hat{x})}{\Phi(\abs{\xi_{\hat{x}}+\xi},\hat{x})}-\frac{f(\hat{y})}{\Phi(\abs{\xi_{\hat{y}}+\xi},\hat{y})}\geq \frac{-L\|f\|_{L^{\infty}(B_{1})}}{\abs{\xi_{\hat{x}}+\xi}^{i(\Phi)}}-\frac{L\|f\|_{L^{\infty}(B_{1})}}{\abs{\xi_{\hat{y}}+\xi}^{i(\Phi)}}\geq -2L\|f\|_{L^{\infty}(B_{1})},
\end{equation}
\begin{equation}\label{2}
	\frac{H(\xi_{\hat{x}}+\xi,\hat{x})}{\Phi(\abs{\xi_{\hat{x}}+\xi},\hat{x})}-\frac{H(\xi_{\hat{y}}+\xi,\hat{y})}{\Phi(\abs{\xi_{\hat{y}}+\xi},\hat{y})}\leq L\left( \frac{\mathcal{K}+\mathcal{M}\abs{\xi_{\hat{x}}+\xi}^{m}}{\abs{\xi_{\hat{x}}+\xi}^{i(\Phi)}}+\frac{\mathcal{K}+\mathcal{M}\abs{\xi_{\hat{y}}+\xi}^{m}}{\abs{\xi_{\hat{y}}+\xi}^{i(\Phi)}} \right) \leq
	2L(\mathcal{K}+\mathcal{M}).
\end{equation}
%\begin{equation}\label{1}
%	\frac{f(\hat{y})}{\Phi\left(\abs{\xi_{\hat{y}}+\xi},\hat{y}\right)}\leq  \frac{L\|f\|_{L^{\infty}(B_{1})}}{\abs{\xi_{\hat{y}}+\xi}^{i(\Phi)}}\leq L\|f\|_{L^{\infty}(B_{1})},
%\end{equation}
%\begin{equation}\label{2}
%	\frac{f(\hat{x})}{\Phi\left(\abs{\xi_{\hat{x}}+\xi},\hat{x}\right)}\geq  \frac{-L\|f\|_{L^{\infty}(B_{1})}}{\abs{\xi_{\hat{x}}+\xi}^{i(\Phi)}}\geq -L\|f\|_{L^{\infty}(B_{1})},
%\end{equation}
%\begin{equation}\label{3}
%	\frac{h(\hat{x})\abs{\xi_{\hat{x}}+\xi}^{m}}{\Phi\left(\abs{\xi_{\hat{x}}+\xi},\hat{x}\right)}\leq \frac{L\|h\|_{L^{\infty}(B_{1})}\abs{\xi_{\hat{x}}+\xi}^{m}}{\abs{\xi_{\hat{x}}+\xi}^{i(\Phi)}}\leq
%	L\|h\|_{L^{\infty}(B_{1})},
%\end{equation}
%\begin{equation}\label{4}
%	\frac{h(\hat{y})\abs{\xi_{\hat{y}}+\xi}^{m}}{\Phi\left(\abs{\xi_{\hat{y}}+\xi},\hat{y}\right)}\leq \frac{L\|h\|_{L^{\infty}(B_{1})}\abs{\xi_{\hat{y}}+\xi}^{m}}{\abs{\xi_{\hat{y}}+\xi}^{i(\Phi)}}\leq
%	L\|h\|_{L^{\infty}(B_{1})}.
%\end{equation}
Substituting \eqref{1}-\eqref{2} into \eqref{bds} and applying \eqref{ps},\eqref{xishu}, we eventually deduce
\begin{equation*}
		-2L\|f\|_{L^{\infty}(B_{1})}
		\leq C_{1}+2L(\mathcal{K}+\mathcal{M})+
		L_{1}\abs{\hat{x}-\hat{y}}^{\beta-1}\left(C_{F}\abs{\hat{x}-\hat{y}}^{\theta-\beta}-4\lambda\omega_{0}\beta(1+\beta)\right),
\end{equation*}
where $C_{1}:=\left(\Lambda(d-1)+\lambda\right)(4L_{2}+2)+C_{F}(2L_{2}+1)$. Choose $L_{2}\geq 4\left(\frac{C_{F}}{2\lambda\omega_{0}\beta(1+\beta)}\right)^{2/(\theta-\beta)}$, we exploit \eqref{youyong} and $\beta<\theta$ to derive
\begin{equation}\label{azuihou11}
C_{F}\Abs{\hat{x}-\hat{y}}^{\theta-\beta}\leq C_{F}\left(\frac{2}{\sqrt{L_{2}}}\right)^{\theta-\beta}\leq 2\lambda\omega_{0}\beta(1+\beta).
\end{equation}
By means of \eqref{azuihou11} %$L_{1}> 1$, %$\|f\|_{L^{\infty}(B_{1})},\|h\|_{L^{\infty}(B_{1})}\leq \delta<1$ 
and $\abs{\hat{x}-\hat{y}}^{\beta-1}<1$, we further arrive at
\begin{equation*}
		-2L(\|f\|_{L^{\infty}(B_{1})}+\mathcal{K}+\mathcal{M})\leq  C_{1}-2L_{1}\lambda\omega_{0}\beta(1+\beta)\abs{\hat{x}-\hat{y}}^{\beta-1}<C_{1}-2L_{1}\lambda\omega_{0}\beta(1+\beta).
\end{equation*}
As a result, if we choose $L_{1}\geq \frac{C_{1}+2L(\|f\|_{L^{\infty}(B_{1})}+\mathcal{K}+\mathcal{M})}{2\lambda\omega_{0}\beta(1+\beta)}$, we reach a contradiction.

{\bf Case 2.} Suppose $0<m\leq i(\Phi)$ and $\abs{\xi}< A_{0}$. In this case, we will use the same arguments as in Proposition \ref{jin1}. In what follows, we mainly explain the differences. Choose $L_{1}>\frac{L_{2}r^{2-\gamma}}{\gamma2^{1-\gamma}}$ so that
\begin{equation*}
	\frac{\gamma L_{1}}{2}\Abs{\hat{x}-\hat{y}}^{\gamma-1}\leq \abs{\xi_{\hat{x}}},\abs{\xi_{\hat{y}}}\leq 2\gamma L_{1}\Abs{\hat{x}-\hat{y}}^{\gamma-1}.
\end{equation*}
To proceed, we select $L_{2}>\left(\frac{2^{4-\gamma}}{\gamma}\right)^{2/(1-\gamma)}$. Then it follows from $\abs{\xi}< A_{0}=3L_{1}$, $\Abs{\hat{x}-\hat{y}}\leq \frac{2}{\sqrt{L_{2}}}$ and $\gamma<1$ that
\begin{equation*}
	\abs{\xi_{\hat{x}}+\xi}\geq \abs{\xi_{\hat{x}}}-\abs{\xi}\geq \frac{\gamma L_{1}}{2}\Abs{\hat{x}-\hat{y}}^{\gamma-1}-A_{0}\geq \frac{\gamma L_{1}}{2}\left(\frac{2}{\sqrt{L_{2}}}\right)^{\gamma-1}-3L_{1}\geq L_{1}>1.
\end{equation*}	
Also, we have $\abs{\xi_{\hat{y}}+\xi}>1$. By analogy with the estimates of \eqref{1}-\eqref{2}, %according to $0<m\leq i(\Phi)$, 
we obtain
\begin{equation*}
	\frac{f(\hat{x})}{\Phi(\abs{\xi_{\hat{x}}+\xi},\hat{x})}-\frac{f(\hat{y})}{\Phi(\abs{\xi_{\hat{y}}+\xi},\hat{y})}\geq -2L\|f\|_{L^{\infty}(B_{1})},
\end{equation*}
\begin{equation*}
	\frac{H(\xi_{\hat{x}}+\xi,\hat{x})}{\Phi(\abs{\xi_{\hat{x}}+\xi},\hat{x})}-\frac{H(\xi_{\hat{y}}+\xi,\hat{y})}{\Phi(\abs{\xi_{\hat{y}}+\xi},\hat{y})}\leq 2L(\mathcal{K}+\mathcal{M}).
\end{equation*}
Substituting the above inequalities into \eqref{ainequality}, we arrive at
\begin{equation*}
	-2L\|f\|_{L^{\infty}(B_{1})}\leq C_{1}+2L(\mathcal{K}+\mathcal{M})+L_{1}\Abs{\hat{x}-\hat{y}}^{\gamma-2}\left(C_{F}\gamma\Abs{\hat{x}-\hat{y}}^{\theta}-4\lambda\gamma(1-\gamma)\right),
\end{equation*}
where $C_{1}:=\left(\Lambda(d-1)+\lambda\right)(4L_{2}+2)+C_{F}(2L_{2}+1)$.
We now select $L_{2}\geq 4\left(\frac{C_{F}}{2\lambda(1-\gamma)}\right)^{2/\theta}$ and apply $\Abs{\hat{x}-\hat{y}}\leq \frac{2}{\sqrt{L_{2}}}$ to derive
$$C_{F}\Abs{\hat{x}-\hat{y}}^{\theta}\leq C_{F}\left(\frac{2}{\sqrt{L_{2}}}\right)^{\theta}\leq 2\lambda(1-\gamma).$$
Then it follows that
\begin{equation*}
	2L_{1}\lambda\gamma(1-\gamma)< C_{1}+2L(\|f\|_{L^{\infty}(B_{1})}+\mathcal{K}+\mathcal{M}).
\end{equation*}
As a consequence, if we choose $L_{1}\geq \frac{C_{1}+2L(\|f\|_{L^{\infty}(B_{1})}+\mathcal{K}+\mathcal{M})}{2\lambda\gamma(1-\gamma)}$, we  reach a contradiction.

{\bf Case 3.} Suppose $-1<i(\Phi)<0$, $0<m\leq 1+i(\Phi)$ and $\xi=0$. We first choose $L_{1}>\max \{4,L_{2}\}$ and $L_{2}\geq \left(2^{2+\beta}\omega_{0}(1+\beta)\right)^{2/\beta}$ so that
$\abs{\xi_{\hat{x}}}< 2L_{1}$ and
\begin{align*}
	\abs{\xi_{\hat{x}}}\geq& L_{1}\omega^{\prime}(\abs{\hat{x}-\hat{y}})-\frac{L_{2}}{2}=L_{1}\left(1-\omega_{0}(1+\beta)\abs{\hat{x}-\hat{y}}^{\beta}\right)-\frac{L_{2}}{2}\\
	\geq&L_{1}\left(1-\omega_{0}(1+\beta)\left(\frac{2}{\sqrt{L_{2}}}\right)^{\beta}\right)-\frac{L_{2}}{2}\geq \frac{3L_{1}}{4}-\frac{L_{2}}{2}\geq \frac{L_{1}}{4}>1.
\end{align*}
Also, we have $1<\abs{\xi_{\hat{y}}}\leq 2L_{1}$.
Applying the assumptions \eqref{A3}, \eqref{A4}, in combination with $-1< i(\Phi)<0$, $0<m\leq 1+i(\Phi)$, and $\Abs{\hat{x}-\hat{y}}^{\theta-1}>1$, we deduce that
\begin{equation*}
	\frac{f(\hat{x})}{\Phi\left(\abs{\xi_{\hat{x}}},\hat{x}\right)}-\frac{f(\hat{y})}{\Phi\left(\abs{\xi_{\hat{y}}},\hat{y}\right)}\geq \frac{-L\|f\|_{L^{\infty}(B_{1})}}{\abs{\xi_{\hat{x}}}^{i(\Phi)}}- \frac{L\|f\|_{L^{\infty}(B_{1})}}{\abs{\xi_{\hat{y}}}^{i(\Phi)}}\geq \frac{-4L\|f\|_{L^{\infty}(B_{1})}}{L_{1}^{i(\Phi)}},
\end{equation*}
%\begin{equation*}
%	\frac{f(\hat{x})}{\Phi\left(\abs{\xi_{\hat{x}}},\hat{x}\right)}\geq  \frac{-L\|f\|_{L^{\infty}(B_{1})}}{\abs{\xi_{\hat{x}}}^{i(\Phi)}}\geq \frac{-2L\|f\|_{L^{\infty}(B_{1})}}{L_{1}^{i(\Phi)}},
%\end{equation*}
\begin{equation*}
	\begin{split}
		\frac{H(\xi_{\hat{x}},\hat{x})}{\Phi\left(\abs{\xi_{\hat{x}}},\hat{x}\right)}-\frac{H(\xi_{\hat{y}},\hat{y})}{\Phi\left(\abs{\xi_{\hat{y}}},\hat{y}\right)}
		&\leq L\bigg( \frac{\mathcal{K}+\mathcal{M}\abs{\xi_{\hat{x}}}^{m}}{\abs{\xi_{\hat{x}}}^{i(\Phi)}}+\frac{\mathcal{K}+\mathcal{M}\abs{\xi_{\hat{y}}}^{m} }{\abs{\xi_{\hat{y}}}^{i(\Phi)}} \bigg)\\
		&\leq 4L\mathcal{K}L_{1}^{-i(\Phi)}+2L\mathcal{M}(2L_{1})^{m-i(\Phi)}\\
		&\leq 	4L\mathcal{K}L_{1}^{-i(\Phi)}+4L\mathcal{M}L_{1}\Abs{\hat{x}-\hat{y}}^{\theta-1}.
	\end{split}
\end{equation*}
%\begin{equation*}
%	\frac{h(\hat{y})\abs{\xi_{\hat{y}}}^{m}}{\Phi\left(\abs{\xi_{\hat{y}}},\hat{y}\right)}\leq \frac{L\|h\|_{L^{\infty}(B_{1})}\abs{\xi_{\hat{y}}}^{m}}{\abs{\xi_{\hat{y}}}^{i(\Phi)}}\leq
%	L\|h\|_{L^{\infty}(B_{1})}(2L_{1})^{m-i(\Phi)}\leq 2L\|h\|_{L^{\infty}(B_{1})}L_{1}\Abs{\hat{x}-\hat{y}}^{\theta-1}.
%\end{equation*}
Substituting the aforementioned inequalities into \eqref{bds}, and combining \eqref{ps} with \eqref{xishu}, we get
\begin{equation*}
	-4L\left(\|f\|_{L^{\infty}(B_{1})}+\mathcal{K}\right)L_{1}^{-i(\Phi)}
	\leq C_{1}+
	 L_{1}\abs{\hat{x}-\hat{y}}^{\beta-1}\left[\left(C_{F}+4L\mathcal{M}\right)\abs{\hat{x}-\hat{y}}^{\theta-\beta}-4\lambda\omega_{0}\beta(1+\beta)\right],
\end{equation*}
where $C_{1}:=\left(\Lambda(d-1)+\lambda\right)(4L_{2}+2)+C_{F}(2L_{2}+1)$. Choosing $L_{2}\geq 4\left(\frac{C_{F}+4L\mathcal{M}}{2\lambda\omega_{0}\beta(1+\beta)}\right)^{2/(\theta-\beta)}$ so that
\begin{equation*}
	\left(C_{F}+4L\mathcal{M}\right)\Abs{\hat{x}-\hat{y}}^{\theta-\beta}\leq \left(C_{F}+4L\mathcal{M}\right)\left(\frac{2}{\sqrt{L_{2}}}\right)^{\theta-\beta}\leq 2\lambda\omega_{0}\beta(1+\beta).
\end{equation*}
Then it follows that
\begin{equation}\label{qiyimaodun}
	2L_{1}\lambda\omega_{0}\beta(1+\beta)\leq 4L\left(\|f\|_{L^{\infty}(B_{1})}+\mathcal{K}\right)L_{1}^{-i(\Phi)}
	 +C_{1}.
\end{equation}
In view of $-1<i(\Phi)<0$, taking $L_{1}$ large enough, we obtain a contradiction with \eqref{qiyimaodun}.
This completes the proof.
\end{proof}
%%%%%%%%%%%%%%%%%%%%%%%%%%%%%%%%%%%%%%%%%%%%%%%%%
%%%%%%%%%%%%%%%%%%%%%%%%%%%%%%%%%%%%%%%%%%%%%%%%%
\section{Tangential path}\label{section4}
This section is solely dedicated to the proof of a key approximation lemma for the degenerate case via compactness arguments, which plays a paramount role in our forthcoming geometric argument.
\begin{lemma}[{\bf Approximation Lemma}]\label{bijin}
	 Suppose the assumptions \eqref{A1}-\eqref{A5} hold true with $i(\Phi)\geq 0$ and 
	$\nu_{0}=\nu_{1}=1$. %and $0<m\leq 1+i(\Phi)$. 
	Let $\xi\in \rn$ be an arbitrarily vector and  $u\in C(B_{1})$ be a normalized viscosity solution of equation \eqref{model111}. Then, for any $\varepsilon>0$, there exists $\sigma\in (0,1)$ depending on $d,\lambda,\Lambda,m,i(\Phi),L,$ and $\varepsilon$ such that if
	\begin{equation*}
	\max\left\{\|{\rm osc}_{{F}}\|_{L^{\infty}\left(B_{1}\right)},\|f\|_{L^{\infty}(B_{1})},\mathcal{K},\mathcal{M}\left(\abs{\xi}^{(m-i(\Phi))_{+}}+1\right) \right\}\leq \sigma,
	\end{equation*}
then one can find $v\in C^{1,\alpha_{0}}(B_{3/4})$, for some $\alpha_{0}\in(0,1)$, satisfying
%which is a viscosity solution a constant coefficient, homogeneous, uniformly $(\lambda,\Lambda)$-elliptic equation\begin{equation}\label{tiaohe}\mathcal{F}(D^{2}v)=0,\quad {\rm in}\quad B_{3/4}\end{equation}such that
	 \begin{equation*}
	 	\|u-v\|_{L^{\infty}(B_{1/2})}\leq \varepsilon.
	 \end{equation*}
Furthermore, $\|v\|_{C^{1,\alpha_{0}}(B_{3/4})}\leq C$, where $C$ depends only on $d,\lambda,\Lambda$.
\end{lemma}
\begin{proof}
	The proof is based on a contradiction argument. If the claim fails, then there exist $\varepsilon_{0}>0$ and sequences
	of functions $\{F_{j}\}_{j\in \mathbb{N}}$, $\{\Phi_{j}\}_{j\in \mathbb{N}}$, $\{H_{j}\}_{j\in \mathbb{N}}$, $\{f_{j}\}_{j\in \mathbb{N}}$, $\{u_{j}\}_{j\in \mathbb{N}}$ and a sequence of vectors $\{\xi_{j}\}_{j\in \mathbb{N}}$ such that
	\begin{itemize}
		\item   [{\rm$({{\rm i}})$}] $u_{j}\in C({B_{1}})$ with $ \|u_{j}\|_{L^{\infty}(B_{1})}\leq 1$ is a viscosity solution of the following equation 
		\begin{equation}\label{model333}
			\Phi_{j}(\abs{Du_{j}+\xi_{j}}, x)F_{j}(D^2 u_{j}, x)+H_{j}(Du_{j}+\xi_{j},x) =f_{j}(x) \quad  \text{in} \quad B_{1},
		\end{equation}
		where $F_{j}:S^{d}\times B_{1}\rightarrow \mathbb{R}$ is uniformly $(\lambda,\Lambda)$-elliptic, $f_{j}\in C({B_{1}})$;
		%$F_{j}:S^{d}\times B_{1}\rightarrow \mathbb{R}$ is uniformly $(\lambda,\Lambda)$-elliptic with $\|{\rm osc}_{{F_{j}}}\|_{L^{\infty}\left(B_{1}\right)}\leq \frac{1}{j}$;
		\item [{\rm$({{\rm ii}})$}] $\Phi_{j}\in C\left([0,\infty)\times B_{1},[0,\infty)\right)$ such that the map $t\mapsto \frac{\Phi_{j}(t,x)}{t^{i(\Phi)}}$ is almost non-decreasing and the map $t\mapsto \frac{\Phi_{j}(t,x)}{t^{s(\Phi)}}$ is almost non-increasing  with the same constant $L\geq 1$, and $\Phi_{j}(1,x)=1$ for all $x\in B_{1}$;
		\item [ {\rm$({{\rm iii}})$}] $H_{j}:\mathbb{R}^{d} \times B_{1}$ is continuous and there exist constants $\mathcal{K}_{j},\mathcal{M}_{j}>0$ such
		that
		\begin{equation}\label{hjtiaojian}
			|H_{j}(t,x)|\leq \mathcal{K}_{j}+\mathcal{M}_{j}|t|^{m} \quad  {\rm for\; every\;}(t,x) \in\mathbb{R}^{d}\times B_{1};
		\end{equation}
	%with $\mathcal{K}_{j},\mathcal{M}_{j}\leq \frac{1}{j}$;
		\item [ {\rm$({{\rm iv}})$}] and
		\begin{equation}\label{fanzhengsmall}
			\max\left\{ \|{\rm osc}_{{F_{j}}}\|_{L^{\infty}\left(B_{1}\right)},\|f_{j}\|_{L^{\infty}(B_{1})}, \mathcal{K}_{j},\mathcal{M}_{j}\left(\abs{\xi_{j}}^{(m-i(\Phi))_{+}}+1\right) \right\}\leq \frac{1}{j}.
		\end{equation}	
		%$f_{j}\in C({B_{1}})$ with $\|f_{j}\|_{L^{\infty}(B_{1})}\leq\frac{1}{j}$;
		%\item [ {\rm$({{\rm v}})$}] $u_{j}\in C({B_{1}})$ with $ \|u_{j}\|_{L^{\infty}(B_{1})}\leq 1$ solves the following equation in the viscosity sense
		%\begin{equation}\label{model333}
		%	\Phi_{j}(\abs{Du_{j}+\xi_{j}}, x)F_{j}(D^2 u_{j}, x)+h_{j}(x)\abs{Du_{j}+\xi_{j}}^{m} =f_{j}(x) \quad  \text{in} \quad B_{1}.
	%	\end{equation}
	\end{itemize}
	Nonetheless, for any $v\in C^{1,\alpha_{0}}\left(B_{3/4}\right)$, it holds
	 \begin{equation*}
		\|u_{j}-v\|_{L^{\infty}(B_{1/2})}>\varepsilon_{0}\quad {\rm for\; any \;} j\in \mathbb{N}.
	\end{equation*}
	
	Since ${F_{j}}$ are uniformly $(\lambda,\Lambda)$-elliptic, they are also Lipschitz continuous in $M$. Thus, it follows from \eqref{fanzhengsmall} and Arzel${\rm \grave{a}}$-Ascoli theorem that there exists some uniformly $(\lambda,\Lambda)$-elliptic
	operator $F_{\infty}$ (with frozen coefficients) such that $F_{j}\rightarrow F_{\infty}$ locally uniformly in $S^{d}$ for all $x\in B_{1}$ fixed, through a subsequence if necessary. In addition, we know from Propositions \ref{jin1} and \ref{jin2} that the sequence $\{u_{j}\}_{j\in\mathbb{N}}\subset C_{loc}^{0,\gamma}(B_{1})$ for some $\gamma\in (0,1)$. Therefore, by applying Arzel${\rm \grave{a}}$-Ascoli theorem again, we conclude that, up to a subsequence,  $u_{j}$ converges locally uniformly in $B_{1}$ to some continuous function $u_{\infty}$ in the $C^{0}$ topology. Particularly, it holds that
	\begin{equation*}
		u_{\infty}\in C({B_{3/4}}) \quad {\rm and}\quad  \|u_{\infty}\|_{L^{\infty}(B_{3/4})}\leq 1,
	\end{equation*}
but
\begin{equation}\label{keymaodun}
	\sup\limits_{x\in B_{1/2}}\;\Abs{u_{\infty}(x)-v(x)}>\varepsilon_{0}.
\end{equation}

In the sequel, our purpose is to verify that the limiting function $u_{\infty}$ is a viscosity solution to the homogeneous equation
	 \begin{equation}\label{homojie}
		F_{\infty}(D^{2}u_{\infty})=0 \quad {\rm in}\quad  B_{3/4}.
	\end{equation}
	For this end, we initially prove that $u_{\infty}$ is a viscosity supersolution. Let $\varphi$ be any test function touching $u_{\infty}$ from below at a point $\overline{x}\in B_{3/4}$, that is,
	$$\varphi(\overline{x})=u_{\infty}(\overline{x})\quad {\rm and}\quad \varphi(x)<u_{\infty}(x)\quad {\rm for\; all}\;\;x\neq \overline{x}.$$
	Without loss of generality, we assume that $\abs{\overline{x}}=u_{\infty}(\overline{x}) = 0$ and $\varphi$ is a quadratic polynomial, namely,
	$$\varphi(x)=\frac{1}{2}\left\langle Mx,x\right\rangle+\left\langle b,x\right\rangle.$$
	Since $u_{j}\rightarrow u_{\infty}$ locally uniformly in $B_{1}$, we see that, for $j$ sufficiently large, the polynomial
	$$\varphi_{j}(x):=\frac{1}{2}\left\langle M(x-x_{j}),x-x_{j}\right\rangle+\left\langle b,x-x_{j}\right\rangle+u_{j}(x_{j})$$
	touches $u_{j}$ from below at $x_{j}$ belonging to a small neighbourhood of zero. Since $u_{j}$ is a viscosity solution of \eqref{model333}, we immediately obtain that
	\begin{equation}\label{sj}
		\Phi_{j}(\abs{b+\xi_{j}}, x_{j})F_{j}(M, x_{j})+
		H(b+\xi_{j},x_{j}) \leq f_{j}(x_{j}).
	\end{equation}
	In what follows, for ease of presentation, we divide the proof into two steps according to the boundedness of the sequence $\{\xi_{j}\}_{j\in\mathbb{N}}$.\\
{\bf Step 1.} If sequence $\{\xi_{j}\}_{j\in \mathbb{N}}$ is unbounded, then we may assume $\abs{\xi_{j}}\rightarrow {\infty}$ as $j\rightarrow\infty$ (up to a subsequence). As a result, there exists $j^{\star}\in\mathbb{N}$ so large that $\abs{\xi_{j}}\geq2\max\{1,\abs{b}\}$ for all $j\geq j^{\star}$. By the triangle inequality, we get
\begin{equation}\label{Step1xiajie}
	\abs{b+\xi_{j}}\geq \abs{\xi_{j}}-|b|\geq \frac{1}{2}\abs{\xi_{j}}\geq 1.
\end{equation}
Combining \eqref{Step1xiajie} with the assumption (ii), \eqref{fanzhengsmall} and $i(\Phi)\geq 0$ yields that
\begin{equation}\label{f1}
	\frac{f_{j}(x_{j})}{\Phi_{j}(\abs{b+\xi_{j}}, x_{j})}\leq  \frac{L\|f_{j}\|_{L^{\infty}(B_{1})}}{\abs{b+\xi_{j}}^{i(\Phi)}}\leq \frac{L2^{i(\Phi)}}{j\abs{\xi_{j}}^{i(\Phi)}}.
\end{equation}
With the help of the assumption (ii) of $\Phi_{j}$ again and the following basic inequality
$$(a_{1}+a_{2})^{m}\leq 2^{m}(a_{1}^{m}+a_{2}^{m}) \quad {\rm for\;all} \;\,a_{1},a_{2}\geq 0,\;m>0,$$
in combination with \eqref{hjtiaojian}, \eqref{fanzhengsmall} and $|\xi_{j}|\geq 2|b|$, we arrive at
\begin{equation*}
	\begin{split}
	\frac{H_{j}\left(b+\xi_{j},x_{j}\right)}{\Phi_{j}(\abs{b+\xi_{j}}, x_{j})}&\leq  \frac{L\left(\mathcal{K}_{j}+\mathcal{M}_{j}	\abs{b+\xi_{j}}^{m}\right)}{\abs{b+\xi_{j}}^{i(\Phi)}}\leq \frac{L2^{i(\Phi)}}{j\abs{\xi_{j}}^{i(\Phi)}}+
	\frac{L2^{i(\Phi)}\mathcal{M}_{j}\abs{b+\xi_{j}}^{m}}{\abs{\xi_{j}}^{i(\Phi)}}\\
	&\leq \frac{L2^{i(\Phi)}}{j\abs{\xi_{j}}^{i(\Phi)}}+L2^{m+i(\Phi)}\frac{\abs{\xi_{j}}^{m-i(\Phi)}}{j\left(\abs{\xi_{j}}^{(m-i(\Phi))_{+}}+1\right)}\left(\Abs{\frac{b}{\xi_{j}}}^{m}+1\right)\\
	&\leq \frac{L2^{i(\Phi)}}{j\abs{\xi_{j}}^{i(\Phi)}}+L2^{m+1+i(\Phi)}\frac{\abs{\xi_{j}}^{m-i(\Phi)}}{j\left(\abs{\xi_{j}}^{(m-i(\Phi))_{+}}+1\right)}.
	\end{split}
\end{equation*}
It follows from $|\xi_{j}|\geq 2$ that
\begin{equation}\label{h1}
\frac{H_{j}\left(b+\xi_{j},x_{j}\right)}{\Phi_{j}(\abs{b+\xi_{j}}, x_{j})}\leq \left\{
\begin{array}{lcl}
	\frac{L2^{i(\Phi)}}{j\abs{\xi_{j}}^{i(\Phi)}}+\frac{L2^{m+i(\Phi)}}{j \abs{\xi_{j}}^{i(\Phi)-m}}& \text{if} &0< m\leq i(\Phi), \\
 \frac{L2^{i(\Phi)}}{j\abs{\xi_{j}}^{i(\Phi)}}+\frac{L2^{m+i(\Phi)}}{j} &\text{if}& i(\Phi)<m\leq 1+ i(\Phi).
\end{array}
\right.
\end{equation}
A combination of \eqref{sj} with \eqref{f1} and \eqref{h1} leads to that
\begin{equation*}
	\begin{split}
		F_{\infty}(M)=\lim\limits_{j\rightarrow \infty}F_{j}(M,x_{j})\leq& \lim\limits_{j\rightarrow \infty}\left(\frac{f_{j}(x_{j})}{\Phi_{j}(\abs{b+\xi_{j}}, x_{j})}- \frac{H_{j}\left(b+\xi_{j},x_{j}\right)}{\Phi_{j}(\abs{b+\xi_{j}}, x_{j})}\right)\\
		\leq &\lim\limits_{j\rightarrow \infty}\left(\Abs{\frac{f_{j}(x_{j})}{\Phi_{j}(\abs{b+\xi_{j}}, x_{j})}}+ \Abs{\frac{H_{j}\left(b+\xi_{j},x_{j}\right)}{\Phi_{j}(\abs{b+\xi_{j}}, x_{j})}}\right)= 0.
	\end{split}
\end{equation*}
{\bf Step 2.} If sequence $\{\xi_{j}\}_{j\in\mathbb{R}^{n}}$ is bounded, then we may assume  $\xi_{j}\rightarrow \xi_{\infty}$ as $j\rightarrow \infty$ (up to a subsequence). As a
consequence, $\xi_{j}+b\rightarrow \xi_{\infty}+b \; {\rm as}\; j\rightarrow \infty.$ At this point, we consider two cases: $\abs{b+\xi_{\infty}}\neq 0$ or $\abs{b+\xi_{\infty}}= 0$. \\
{\bf Case 1.} $\abs{b+\xi_{\infty}}\neq 0$. Note that $\Abs{\xi_{j}+b}\geq \frac{1}{2}\Abs{\xi_{\infty}+b}$ for $j$ large enough.
Thereby, applying the assumption (ii) of $\Phi_{j}$, $i(\Phi)\geq 0$, and the aforementioned inequality, for $j$ sufficiently large, we have
\begin{equation*}
	\begin{cases}
			\Phi_{j}(\abs{b+\xi_{j}}, x_{j})\geq L^{-1}\abs{b+\xi_{j}}^{i(\Phi)}\geq 2^{-i(\Phi)}L^{-1}\abs{b+\xi_{\infty}}^{i(\Phi)} & \text{if}\;\; \Abs{\xi_{\infty}+b}\geq 1, \\
		\Phi_{j}(\abs{b+\xi_{j}}, x_{j})\geq L^{-1}\abs{b+\xi_{j}}^{s(\Phi)}\geq 2^{-s(\Phi)}L^{-1}\abs{b+\xi_{\infty}}^{s(\Phi)} & \text{if}\;\; \Abs{\xi_{\infty}+b}<1.
	\end{cases}
\end{equation*}
Taking into account the aforementioned estimates, and combining with the conditions (ii)-(iv) and $0\leq i(\Phi)\leq s(\Phi)$, we deduce that
\begin{equation}\label{f2}
	\frac{f_{j}(x_{j})}{\Phi_{j}(\abs{b+\xi_{j}}, x_{j})}\leq \frac{2^{s(\Phi)}L}{j\min\left\{\abs{b+\xi_{\infty}}^{i(\Phi)},\abs{b+\xi_{\infty}}^{s(\Phi)}\right\}},
\end{equation}
\begin{equation}\label{h2}
	\begin{split}
		\frac{H_{j}\left(b+\xi_{j},x_{j}\right)}{\Phi_{j}(\abs{b+\xi_{j}}, x_{j})}&\leq \frac{2^{s(\Phi)}L}{j\min\left\{\abs{b+\xi_{\infty}}^{i(\Phi)},\abs{b+\xi_{\infty}}^{s(\Phi)}\right\}}\bigg(1+ \frac{\abs{b+\xi_{j}}^{m}}{\abs{\xi_{j}}^{(m-i(\Phi))_{+}}+1}\bigg)
		\\
		&\leq \frac{2^{s(\Phi)}L\left(1+\abs{b+\xi_{j}}^{m}\right)}{j\min\left\{\abs{b+\xi_{\infty}}^{i(\Phi)},\abs{b+\xi_{\infty}}^{s(\Phi)}\right\}}.
	\end{split}
\end{equation}
Combining \eqref{sj} with \eqref{f2} and \eqref{h2}, we arrive at
\begin{equation*}
		F_{\infty}(M)=\lim\limits_{j\rightarrow \infty}F_{j}(M,x_{j})\leq\lim\limits_{j\rightarrow \infty}\left(\Abs{\frac{f_{j}(x_{j})}{\Phi_{j}(\abs{b+\xi_{j}}, x_{j})}}+ \Abs{\frac{H_{j}\left(b+\xi_{j},x_{j}\right)}{\Phi_{j}(\abs{b+\xi_{j}}, x_{j})}}\right)= 0.
\end{equation*}
{\bf Case 2.} $\Abs{b+\xi_{\infty}}= 0$. By contradiction, let us assume that
\begin{equation}\label{contracdiction}
	F_{\infty}(M)>0.
\end{equation}
The ellipticity condition of $F_{\infty}$ implies that matrix $M$ has at least one positive eigenvalue. Let $\rn=T\oplus Q$ be the orthogonal sum, where $T:={\rm span}\{e_{1},e_{2},...,e_{k}\}$ is the invariant space composed of those eigenvectors corresponding to
positive eigenvalues of $M$, and $Q:=\{y\in \rn:\left\langle y,\eta  \right\rangle=0\;{\rm for\;all}\;\eta\in T\}$. Since $\Abs{\xi_{\infty}+b}=0$, two situations
can occur: $b=-\xi_{\infty}$ with $|b|,\abs{\xi_{\infty}}>0$ or $|b|=\abs{\xi_{\infty}}=0$.

{\bf Case 2.1.} $b=-\xi_{\infty}$ with $|b|,\abs{\xi_{\infty}}>0$. Let $\gamma>0$ and
\begin{equation}\label{model3}
	\varphi_{\gamma}(x):=\varphi(x)+\gamma \Abs{P_{T}(x)}=\frac{1}{2}\left\langle Mx,x\right\rangle+\left\langle b,x\right\rangle+\gamma \Abs{P_{T}(x)},
\end{equation}
where $P_{T}$ stands for the orthogonal projection over $T$. Since $u_{j}\rightarrow u_{\infty}$ locally uniformly in $B_{1}$ and $\varphi$ touches $u$ from below at the origin, then, for $\gamma$ small enough, $\varphi_{\gamma}$ touches
$u_{j}$ from below at a point $x_{j}^{\gamma}$ belonging to a small neighbourhood of 0. Moreover, there holds that, up to a subsequence, $x_{j}^{\gamma}\rightarrow x_{*}$ for some $x_{*}\in B_{3/4}$ as $j\rightarrow\infty$. At this point we examine two scenarios: $P_{T}\left(x_{j}^{\gamma}\right)=0$ and $P_{T}\left(x_{j}^{\gamma}\right)\neq 0$.

First, we consider $P_{T}\left(x_{j}^{\gamma}\right)=0$.
Notice that
\begin{equation*}
	\tilde{\varphi}_{\gamma}(x):=\varphi(x)+\gamma \left\langle e,P_{T}(x)\right\rangle
\end{equation*}
touches $u_{j}$ from below at $x_{j}^{\gamma}$ for every $e\in \mathbb{S}^{d-1}$ (i.e., $|e|=1$). Through a direct calculation, we derive
\begin{equation*}
	D\tilde{\varphi}_{\gamma}(x_{j}^{\gamma})=Mx_{j}^{\gamma}+b+\gamma P_{T}(e),\quad D^{2}\tilde{\varphi}_{\gamma}(x_{j}^{\gamma})=M.
\end{equation*}
We choose $e\in T\cap \mathbb{S}^{d-1}$ such that $P_{T}(e)=e$. Then by means of viscosity inequality \eqref{model333}, we get
\begin{equation}\label{model8}
	\Phi_{j}(\abs{Mx_{j}^{\gamma}+b+\gamma e+\xi_{j}}, x_{j}^{\gamma})F_{j}(M, x_{j}^{\gamma})+H_{j}\left(Mx_{j}^{\gamma}+b+\gamma e+\xi_{j},x_{j}^{\gamma}\right) \leq f_{j}(x_{j}^{\gamma}).
\end{equation}
If $Mx_{*}=0$, then for $j$ sufficiently large, we have
\begin{equation*}
	\Abs{Mx_{j}^{\gamma}+b+\xi_{j}}\leq \frac{\gamma}{2}.
\end{equation*}
This along with the triangle inequality yields
\begin{equation}\label{xiaoyu1}
	\frac{\gamma}{2}\leq \Abs{Mx_{j}^{\gamma}+b+\gamma e+\xi_{j}}\leq  \frac{3\gamma}{2}.
\end{equation}
By analogy with the estimates of \eqref{f2} and \eqref{h2}, we can deduce that
\begin{equation*}
		\frac{f_{j}(x_{j}^{\gamma})}{\Phi_{j}(\abs{Mx_{j}^{\gamma}+b+\gamma e+\xi_{j}}, x_{j}^{\gamma})}\leq \frac{2^{s(\Phi)}L}{j\gamma^{s(\Phi)}},
\end{equation*}
\begin{equation*}
	\frac{H_{j}\left(Mx_{j}^{\gamma}+b+\gamma e+\xi_{j},x_{j}^{\gamma}\right)}{\Phi_{j}(\abs{Mx_{j}^{\gamma}+b+\gamma e+\xi_{j}}, x_{j}^{\gamma})}\leq
		\frac{2^{s(\Phi)}L}{j\gamma^{s(\Phi)}}+\frac{2^{s(\Phi)-m}L(3\gamma)^{m}}{j\gamma^{s(\Phi)}}.
\end{equation*}
We further get
\begin{equation*}
		F_{\infty}(M)=\lim\limits_{j\rightarrow \infty}F_{j}(M,x_{j}^{\gamma})\leq \lim\limits_{j\rightarrow \infty}\frac{\abs{f_{j}(x_{j}^{\gamma})}+\Abs{H_{j}\left(Mx_{j}^{\gamma}+b+\gamma e+\xi_{j},x_{j}^{\gamma}\right)}}{\Phi_{j}(\abs{Mx_{j}^{\gamma}+b+\gamma e+\xi_{j}}, x_{j}^{\gamma})}= 0.
\end{equation*}
On the other hand, if $\Abs{Mx_{*}}>0$, we start off by considering the case in which $T\equiv \rn$ and select $e\in \mathbb{S}^{d-1}$ such that
\begin{equation*}
	\Abs{Mx_{*}+\gamma P_{T}(e)}=\Abs{Mx_{*}+\gamma e}>0.
\end{equation*}
For $j$ large enough, we have
\begin{equation}\label{case1}
	\Abs{Mx_{j}^{\gamma}+\gamma e}\geq \frac{1}{2}\Abs{Mx_{*}+\gamma e}\quad {\rm and}\quad \abs{\xi_{j}+b}\leq \frac{1}{8}\Abs{Mx_{*}+\gamma e}.
\end{equation}
Furthermore, if $T\neq \rn$, then we choose $e\in \mathbb{S}^{d-1}\cap T^{\perp}$ such that
\begin{equation*}
	\Abs{Mx_{*}+\gamma P_{T}(e)}=\Abs{Mx_{*}}>0.
\end{equation*}
Again for $j$ large enough, we have
\begin{equation}\label{case2}
	\Abs{Mx_{j}^{\gamma}}\geq \frac{1}{2}\Abs{Mx_{*}}\quad {\rm and}\quad \abs{\xi_{j}+b}\leq \frac{1}{8}\Abs{Mx_{*}}.
\end{equation}
Thus, using either \eqref{case1} or \eqref{case2}, we arrive at
\begin{equation*}
	\Abs{Mx_{j}^{\gamma}+b+\gamma P_{T}(e)+\xi_{j}}>\frac{1}{4}\Abs{Mx_{*}+\gamma P_{T}(e)}>0.
\end{equation*}
By the same arguments as before, we can also conclude that $F_{\infty}(M)\leq 0$.

Next, let us consider $P_{T}\left(x_{j}^{\gamma}\right)\neq 0$. Note that $\Abs{P_{T}(x)}$ is smooth and convex in a small neighbourhood of $x_{j}^{\gamma}$. Because of $P_{T}$ being a projection, then
\begin{equation}\label{feifuding}
\Abs{P_{T}(x)}D\left(\Abs{P_{T}(x)}\right)=P_{T}(x)\quad {\rm and} \quad D^{2}\left(\Abs{P_{T}(x)}\right)\;\,{\rm is \;nonnegative\; definite}.
\end{equation}
Hence, we arrive at
\begin{equation*}
		\Phi_{j}(\abs{Mx_{j}^{\gamma}+b+\gamma \zeta_{j}^{\gamma}+\xi_{j}}, x_{j}^{\gamma})F_{j}(M+\gamma D^{2}\left(\Abs{P_{T}(x_{j}^{\gamma})}\right), x_{j}^{\gamma})
		+H_{j}\left(Mx_{j}^{\gamma}+b+\gamma \zeta_{j}^{\gamma}+\xi_{j},x_{j}^{\gamma}\right) \leq f_{j}(x_{j}^{\gamma}).
\end{equation*}
where 	$\zeta_{j}^{\gamma}:=\frac{{P_{T}(x_{j}^{\gamma})}}{\Abs{P_{T}(x_{j}^{\gamma})}}$.
Observe that $\Abs{\zeta_{j}^{\gamma}}=1$. Set $e:=\zeta_{j}^{\gamma}$, we can perform the same procedure as in the case $P_{T}\left(x_{j}^{\gamma}\right)= 0$ via distinguishing $Mx_{*}=0$ and $Mx_{*}\neq 0$. Then we can conclude that
$$F_{\infty}(M+\gamma D^{2}\left(\Abs{P_{T}(x_{*})}\right))\leq 0.$$
By virtue of \eqref{feifuding} and the ellipticity condition on $F_{\infty}$, we derive $$F_{\infty}(M)\leq F_{\infty}(M+\gamma D^{2}\left(\Abs{P_{T}(x_{*})}\right))\leq0,$$ which contradicts the assumption \eqref{contracdiction}.

{\bf Case 2.2.} $b=\xi_{\infty}=0$. In this case, the procedures become easier. Since $\frac{1}{2}\left\langle Mx,x\right\rangle$ touches $u_{\infty}$
from below at the origin and $u_{j}\rightarrow u_{\infty}$ locally uniformly, then the test function
$$\hat{\varphi}(x):=\frac{1}{2}\left\langle Mx,x\right\rangle+\gamma \Abs{P_{T}(x)}$$
touches $u_{j}$ from below at a point $\hat{x}_{j}^{\gamma}\in B_{r}$ for $\gamma>0$ sufficiently small. Likewise, we will analyze two scenarios: $\Abs{P_{T}(\hat{x}_{j}^{\gamma})}=0$ and $\Abs{P_{T}(\hat{x}_{j}^{\gamma})}\neq 0$, which is in a similar manner to Case 2.1. Eventually, we also conclude $F_{\infty}(M)\leq 0$.

At this stage, we have shown
that $u_{\infty}$ is a viscosity supersolution of \eqref{homojie}. In order
to verify that $u_{\infty}$ is a viscosity subsolution of \eqref{homojie}, it suffices to show that $-u_{\infty}$ is a supersolution to $\hat{F}_{\infty}(D^{2}w)=0$, where $\hat{F}_{\infty}(M):=-{F}_{\infty}(-M)$ is uniformly $(\lambda,\Lambda)$-elliptic as well. It now follows from the well-known regularity results in \cite[
Chapter 5]{Caff1} that $u_{\infty}\in C_{loc}^{1,\alpha_{0}}(B_{3/4})$ for some $\alpha_{0}\in(0,1)$ and that $\|u_{\infty}\|_{C^{1,\alpha_{0}}\left(B_{1/2}\right)}\leq C(d,\lambda,\Lambda)$. Finally, taking $v=u_{\infty}$, we reach a contradiction with \eqref{keymaodun}. This completes the proof of the desired result.
\end{proof}
%%%%%%%%%%%%%%%%%%%%%%%%%%%%%%%%%%%%%%%%%%%%%%%%%%%%%%%%%%%%%%%%%%%%%%%%%%%%%%%%%%%%%%%%%%%%%%%%%%%%%%%%
\section{Proof of Theorem \ref{main}}\label{section5}
In this section, we present the core sharp oscillation decay result, which ultimately leads to the optimal $C^{1,\alpha}$ regularity estimate for solutions to \eqref{model}. The following lemma provides the first geometric iteration, which follows immediately from the
approximation lemma.
\begin{lemma}\label{diedai1}
	Suppose that the hypotheses \eqref{A1}-\eqref{A5} hold true with $i(\Phi)\geq 0$ and
	$\nu_{0}=\nu_{1}=1$. %and $0<m\leq 1+i(\Phi)$. 
	Let $\xi\in \rn$ be an arbitrarily vector and  $u\in C(B_{1})$ be a normalized viscosity solution to \eqref{model111}. Given
	\begin{equation}\label{zhibiao}
		\alpha\in (0,\alpha_{0})\cap \left(0,\frac{1}{1+s(\Phi)}\right],
	\end{equation}
	there exists constant $\delta>0$ depending
	on $d,\lambda,\Lambda,\alpha,L,m$ and $i(\Phi)$, such that if
	\begin{equation*}
		\max\left\{\|{\rm osc}_{{F}}\|_{L^{\infty}\left(B_{1}\right)},\|f\|_{L^{\infty}(B_{1})},\mathcal{K},\mathcal{M}\left(\abs{\xi}^{(m-i(\Phi))_{+}}+1\right)\right\}\leq \delta,
	\end{equation*}
	then there exist constant $0<\rho<\frac{1}{2}$ depending on $d,\lambda,\Lambda$ and $\alpha$, and an affine function $l(x)=a+b\cdot x$ $(a\in\mathbb{R},b\in\rn)$ such that
	\begin{equation*}
		\|u-l\|_{L^{\infty}(B_{\rho})}\leq \rho^{1+\alpha}
	\end{equation*}
	and
	\begin{equation*}
		|a|+|b|\leq C(d,\lambda,\Lambda).
	\end{equation*}
\end{lemma}
\begin{proof}
For $\varepsilon>0$ to be fixed a posteriori, let $v$ be the approximate function from Lemma \ref{bijin}, which is $\varepsilon$-close to $u$ in $L^{\infty}(B_{1/2})$. Remember that Lemma \ref{bijin} guarantees the existence of such a function $v$, provided that $\delta>0$ is small enough.

Now, according to the $C^{1,\alpha_{0}}$-regularity of $v$, we have
\begin{equation}\label{zhengzexing}
	\sup\limits_{x\in B_{\rho}}\Abs{v(x)-\left(v(0)+Dv(0)\cdot x\right)}\leq C\rho^{1+\alpha_{0}}\quad {\rm for \;all\;} \rho\in\left(0,\frac{1}{2}\right),
\end{equation}
	\begin{equation*}
	|v(0)|+|Dv(0)|\leq C,
\end{equation*}
where constants $\alpha_{0}\in(0,1)$ and $C>0$ depend only on $d,\lambda,\Lambda$. At this point, let us denote
	\begin{equation*}
		l(x):=a+b\cdot x:=v(0)+Dv(0)\cdot x.
	\end{equation*}
To proceed, a combination of the triangle inequality  with Lemma \ref{bijin} and \eqref{zhengzexing} yields that
\begin{equation}\label{5.1}
	\sup\limits_{x\in B_{\rho}}\Abs{u(x)-l(x)}\leq \sup\limits_{x\in B_{\rho}}\Abs{u(x)-v(x)}+\sup\limits_{x\in B_{\rho}}\Abs{v(x)-l(x)}\leq \varepsilon+C\rho^{1+\alpha_{0}}.
\end{equation}	
	In light of $\alpha<\alpha_{0}$, we choose $0<\rho\ll 1$ such that
	\begin{equation}\label{5.2}
		C\rho^{\alpha_{0}-\alpha}\leq \frac{1}{2}.%\Leftrightarrow\rho\leq \left(\frac{1}{2C}\right)^{\frac{1}{\alpha_{0}-\alpha}}.
	\end{equation}
Also, for the previous selected radius, we fix
	\begin{equation}\label{5.3}
	\varepsilon:=\frac{1}{2}\rho^{1+\alpha}.
\end{equation}	
Finally, combining with \eqref{5.1}-\eqref{5.3}, we obtain
\begin{equation*}
	\|u-l\|_{L^{\infty}(B_{\rho})}\leq \rho^{1+\alpha}.
\end{equation*}
This completes the proof of the desired result.
\end{proof}
Next, we iterate the previous result in a sequence of concentric, shrinking balls.
\begin{lemma}[\bf Geometric iterations]\label{diedai2}
			Suppose that the assumptions \eqref{A1}-\eqref{A5} hold true with $i(\Phi)\geq 0$ and 
			$\nu_{0}=\nu_{1}=1$. % and $0<m\leq 1+i(\Phi)$. 
			 Let $u\in C(B_{1})$ be a normalized viscosity solution to \eqref{model}. Given $\alpha$ as in \eqref{zhibiao}, there are constants $\rho\in(0,\frac{1}{2})$ and $\delta>0$, both of which are the same as those in Lemma \ref{diedai1}, such that if
		\begin{equation}\label{5.6}
			\max\left\{\|{\rm osc}_{{F}}\|_{L^{\infty}\left(B_{1}\right)},\|f\|_{L^{\infty}(B_{1})},\mathcal{K},\mathcal{M}\right\} \leq \delta,
		\end{equation} then for any $j\in\mathbb{N}$, there exists a sequence of affine functions $\{l_{j}(x)\}_{j\in\mathbb{N}}$, where   $l_{j}(x)=a_{j}+b_{j}\cdot x$ $(a_{j}\in\mathbb{R},b_{j}\in\rn)$, fulfilling
	\begin{equation}\label{guina1}
		\|u-l_{j}\|_{L^{\infty}(B_{\rho^{j}})}\leq \rho^{j(1+\alpha)}
	\end{equation}
	and
	\begin{equation}\label{guina2}
		|a_{j}-a_{j-1}|+\rho^{j-1}|b_{j}-b_{j-1}|\leq C\rho^{(j-1)(1+\alpha)}.
	\end{equation}
where the constant $C$ depends only on $d,\lambda,\Lambda$.
\end{lemma}
\begin{proof}
The proof follows from an induction argument. Let $a_{0}=0$, $b_{0}=0$, and $l_{1}$ be given by Lemma \ref{diedai1}. Then the claim immediately holds for $j=1$ by Lemma \ref{diedai1}. Suppose that the conclusion holds true for $j=1,2,...,k$. Now we are going to show that the claim also holds for $j=k+1$. To this end, we introduce an auxiliary function $u_{k}:B_{1}\rightarrow \mathbb{R}$ as
	\begin{equation*}
	u_{k}(x):=\frac{u\left(\rho^{k}x\right)-l_{k}\left(\rho^{k}x\right)}{\rho^{k(1+\alpha)}}.
\end{equation*}
We can readily check that $u_{k}$ solves
\begin{equation} \label{Eq2.1}
	\Phi_{k}(\Abs{D{u_{k}}+\xi_{k}}, x){F}_{k}(D^2 {u_{k}}, x)+H_{k}(D{u_{k}}+\xi_{k}, x)	= {f_{k}}(x) \quad \text{in} \quad  B_{1}
\end{equation}
in the viscosity sense, where
\begin{align*}
	{F_{k}}(X,x):=&\rho^{k(1-\alpha)}F\left(\rho^{k(\alpha-1)}X,\rho^{k}x\right),\\
	{\Phi_{k}}(t,x):=&\frac{\Phi\left(\rho^{k\alpha}t,\rho^{k}x\right)}{\Phi\left(\rho^{k\alpha},\rho^{k}x\right)},\\
	H_{k}(t,x):=&\frac{\rho^{k(1-\alpha)}}{\Phi\left(\rho^{k\alpha},\rho^{k}x\right)}H\left(\rho^{k\alpha}t,\rho^{k}x\right),\\
	%{h_{k}}(x):=&\frac{\rho^{k(1-\alpha(1-m)))}}{\Phi\left(\rho^{k\alpha},\rho^{k}x\right)}h(\rho^{k}x),\\
	{f_{k}}(x):=&\frac{\rho^{k(1-\alpha)}}{\Phi\left(\rho^{k\alpha},\rho^{k}x\right)}f(\rho^{k}x),\\
	\xi_{k}:=&\rho^{-k\alpha}b_{k}.
\end{align*}		
It is standard to verify that ${F_{k}}$ is also a uniformly $(\lambda,\Lambda)$-elliptic operator. Through a straightforward calculation, we get
\begin{equation*}
	{\rm osc}_{{F_{k}}}(x,0)=\sup\limits_{M\in S^{d}\setminus \{0\}}\frac{\Abs{F\left(\rho^{k(\alpha-1)}M,\rho^{k}x\right)-F\left(\rho^{k(\alpha-1)}M,0\right)}}{\rho^{k(\alpha-1)}\|M\|}={\rm osc}_{{F}}(\rho^{k}x,0),
\end{equation*}
which together with \eqref{5.6} yields that
\begin{equation*}
	\|{\rm osc}_{{F_{k}}}\|_{L^{\infty}\left(B_{1}\right)}= 	\|{\rm osc}_{{F}}\|_{L^{\infty}\left(B_{\rho^{k}}\right)}\leq \|{\rm osc}_{{F}}\|_{L^{\infty}\left(B_{1}\right)}\leq \delta.
\end{equation*} 
By induction assumption, we have
$$\|{u_{k}}\|_{L^{\infty}\left(B_{1}\right)}\leq 1.$$
Note that $\Phi_{k}$ still satisfies the properities that the map $t\mapsto \frac{{\Phi_{k}}(t,x)}{t^{i(\Phi)}}$ is almost non-decreasing, the map $t\mapsto \frac{{\Phi_{k}}(t,x)}{t^{s(\Phi)}}$ is almost non-increasing with the same constant
$L\geq 1$ and ${\Phi_{k}}(1,x)=1$ for all $x\in B_{1}$.
In addition, combining the properties of $\Phi_{k}$ with \eqref{A4} and $\rho\in(0,\frac{1}{2})$, we get
\begin{equation*}
	|H_{k}(t,x)|\leq \frac{L\rho^{k(1-\alpha)}}{\rho^{k\alpha s(\Phi)}}\left(\mathcal{K}+\mathcal{M}\rho^{k\alpha m}|t|^{m}\right)=:\mathcal{K}_{k}+\mathcal{M}_{k}|t|^{m}.
\end{equation*} 
By virtue of the properties of $\Phi_{k}$ again, combining \eqref{5.6} with $\alpha\in \left(0,\frac{1}{1+s(\Phi)}\right]$, we deduce that
%Also, by virtue of the properities of $\Phi_{k}$, \eqref{5.6}, $\rho\in(0,\frac{1}{2})$ and $\alpha\in \left(0,\frac{1}{1+s(\Phi)}\right]$, one easily deduce that
\begin{equation*}
	\|{f_{k}}\|_{L^{\infty}\left(B_{1}\right)}\leq \frac{L\rho^{k(1-\alpha)}\|{f}\|_{L^{\infty}\left(B_{1}\right)}}{\rho^{k\alpha s(\Phi)}}\leq L\delta \rho^{k\left(1-\alpha(1+s(\Phi))\right)}\leq L\delta,
\end{equation*}
\begin{equation*}
	\mathcal{K}_{k}= L \rho^{k\left(1-\alpha(1+s(\Phi))\right)}\mathcal{K}\leq L\delta.
\end{equation*}
%\begin{equation}\label{hk}
%	\|{h_{k}}\|_{L^{\infty}\left(B_{1}\right)}\leq \frac{L\rho^{k(1-\alpha(1-m)))}}{\rho^{k\alpha s(\Phi)}}\|{h}\|_{L^{\infty}\left(B_{1}\right)}\leq L\delta \rho^{k\left(1-\alpha(1+s(\Phi)-m)\right)}.
%\end{equation}
%By means of $\alpha\in \left(0,\frac{1}{1+s(\Phi)}\right]$ again, we obtain
%\begin{equation*}
%	\mathcal{K}_{k}\leq L \rho^{k\left(1-\alpha(1+s(\Phi))\right)}\mathcal{K}\leq L\delta,\quad \mathcal{M}_{k}\leq L \rho^{k\left(1-\alpha(1+s(\Phi)-m)\right)}\mathcal{M}\leq L\delta.
%\end{equation*}
Now we analyze the quantity $\mathcal{M}_{k}\left(\abs{\xi_{k}}^{(m-i(\Phi))_{+}}+1\right)$. Applying \eqref{5.6} and $\xi_{k}=\rho^{-k\alpha}b_{k}$ to obtain
%Combining \eqref{hk} with $\xi_{k}=\rho^{-k\alpha}b_{k}$, we obtain
\begin{equation}\label{Hamilliang}
	\mathcal{M}_{k}\left(\abs{\xi_{k}}^{(m-i(\Phi))_{+}}+1\right)\leq L\delta \rho^{k\left(1-\alpha(1+s(\Phi)-m)\right)-k\alpha(m-i(\Phi))_{+}}\left(\abs{b_{k}}^{(m-i(\Phi))_{+}}+1\right).
\end{equation}		
If $0<m\leq i(\Phi)$,
then by taking advantage of \eqref{Hamilliang} and the fact that $1-\alpha(1+s(\Phi)-m)>0$, we get
$$\mathcal{M}_{k}\left(\abs{\xi_{k}}^{(m-i(\Phi))_{+}}+1\right)\leq 2L\delta \rho^{k\left(1-\alpha(1+s(\Phi)-m)\right)}\leq 2L\delta.$$
On the other hand, if $i(\Phi)<m\leq 1+i(\Phi)$, then
\begin{align*}
	\mathcal{M}_{k}\left(\abs{\xi_{k}}^{(m-i(\Phi))_{+}}+1\right)\leq L\delta \rho^{k\left(1-\alpha(1+s(\Phi)-i(\Phi))\right)}\left(\abs{b_{k}}^{m-i(\Phi)}+1\right)
	\leq L\delta\left(\abs{b_{k}}^{m-i(\Phi)}+1\right),
\end{align*}
where we use the fact that $1-\alpha(1+s(\Phi)-i(\Phi))>0$ in the last inequality.
It follows from the induction hypothesis \eqref{guina2} and $\rho<\frac{1}{2}$ that		
\begin{align*}
	|b_{k}|\leq |b_{0}|+\sum_{j=1}^{k}\abs{b_{j}-b_{j-1}}\leq C\sum_{j=1}^{k}\rho^{\alpha(j-1)}\leq \frac{C}{1-\rho^{\alpha}}\leq 2C.
\end{align*}
In summary, for $0<m\leq 1+i(\Phi)$, we have
\begin{align*}
	\mathcal{M}_{k}\left(\abs{\xi_{k}}^{(m-i(\Phi))_{+}}+1\right)\leq L\delta \left((2C)^{(m-i(\Phi))_{+}}+1\right).
\end{align*}
Now we select $\delta>0$ small enough so that
\begin{equation*}
 L\delta \left((2C)^{(m-i(\Phi))_{+}}+1\right)\leq \sigma
\end{equation*}	
with the $\sigma$ appearing in the statement of Lemma \ref{bijin}. Note that once we fix the value of $\varepsilon$ as in \eqref{5.3}, the quantity $\sigma$ in Lemma \ref{bijin} is determined accordingly. As a consequence, we arrive at
\begin{equation*}
	\max\left\{\|{\rm osc}_{{F_{k}}}\|_{L^{\infty}\left(B_{1}\right)},\|{f_{k}}\|_{L^{\infty}\left(B_{1}\right)},\mathcal{K}_{k},\mathcal{M}_{k}\left(\abs{\xi_{k}}^{(m-i(\Phi))_{+}}+1\right)\right\}\leq \sigma.
\end{equation*}	
At this moment, the assumptions in Lemma \ref{bijin} are satisfied. Thus, we can apply Lemma \ref{diedai1} to $u_{k}$ and obtain
\begin{equation*}
	\|u_{k}-\tilde{l}\|_{L^{\infty}(B_{\rho})}\leq \rho^{1+\alpha},
\end{equation*}
where $\tilde{l}(x)$ is an affine function of the form
 $\tilde{l}(x)=\tilde{a}+\tilde{b}\cdot x$ with $|\tilde{a}|+|\tilde{b}|\leq C(d,\lambda,\Lambda)$. In the sequel, we define the  approximating affine function $l_{k+1}$ as
\begin{equation*}
	l_{k+1}(x):=a_{k+1}+b_{k+1}\cdot x,
\end{equation*}
where
\begin{equation*}
	a_{k+1}:=a_{k}+\rho^{k(1+\alpha)}\tilde{a}\quad {\rm and}\quad     b_{k+1}:=b_{k}+\rho^{k\alpha}\tilde{b}.
\end{equation*}
Scaling back, we reach that
\begin{equation*}
	\|u-l_{k+1}\|_{L^{\infty}(B_{\rho^{k+1}})}=\rho^{k(1+\alpha)}\|u_{k}-\tilde{l}\|_{L^{\infty}(B_{\rho})}
	\leq \rho^{(k+1)(1+\alpha)}
\end{equation*}
	and
\begin{equation*}
	|a_{k+1}-a_{k}|+\rho^{k}|b_{k+1}-b_{k}|=\rho^{k(1+\alpha)}\left(\Abs{\tilde{a}}+\Abs{\tilde{b}}\right)\leq C\rho^{k(1+\alpha)}.
\end{equation*}
This completes the proof of the desired result.
\end{proof}
\begin{corollary}\label{small11}
	Suppose that the assumptions of Lemma \ref{diedai2} are in force. Then, there exists an affine function   $\overline{l}(x)=\overline{a}+\overline{b}\cdot x$ $(\overline{a}\in\mathbb{R},\overline{b}\in\rn)$ with
	$$|\overline{a}|+|\overline{b}|\leq C$$
	such that for each $0<r\leq \rho$ with $\rho$ being identical to that in Lemma \ref{diedai2}	
\begin{equation*}
	\|u-\overline{l}\|_{L^{\infty}(B_{r})}\leq Cr^{1+\alpha},
\end{equation*}
where the constant $C$ depends on $d,\lambda,\Lambda$ and $\alpha$.
\end{corollary}
\begin{proof}
	From Lemma \ref{diedai2}, we know that $\{a_{j}\}_{j\in\mathbb{N}}\subset \mathbb{R}$, $\{b_{j}\}_{j\in\mathbb{N}}\subset \rn$ are Cauchy sequences, hence, they converge. Now, we denote
\begin{equation*}
	\overline{a}:=\lim\limits_{j\rightarrow\infty}a_{j},
	\quad \overline{b}:=\lim\limits_{j\rightarrow\infty}b_{j},
\end{equation*}	
$$\overline{l}(x):=\overline{a}+\overline{b}\cdot x.$$
In the sequel, for any $n\geq j$, a combination of the triangle inequality and \eqref{guina2} yields that
\begin{align*}
	|a_{j}-a_{n}|\leq \sum_{k=j}^{n-1}\abs{a_{k}-a_{k+1}}\leq C\sum_{k=j}^{n-1}\rho^{k(1+\alpha)}\leq C\rho^{j(1+\alpha)} \frac{1-\rho^{(n-j)(1+\alpha)}}{1-\rho^{(1+\alpha)}}.
\end{align*}	
Letting $n\rightarrow \infty$, we get	
\begin{align}\label{511}
	|a_{j}-\overline{a}|\leq  \frac{C\rho^{j(1+\alpha)}}{1-\rho^{(1+\alpha)}}.
\end{align}	
In a similar way, we deduce
\begin{align}\label{512}
	|b_{j}-\overline{b}|\leq  \frac{C\rho^{j\alpha}}{1-\rho^{\alpha}}.
\end{align}	

We now claim that
\begin{equation*}
	\|u-\overline{l}\|_{L^{\infty}(B_{r})}\leq Cr^{1+\alpha}, \quad \forall r\in (0,\rho].
\end{equation*}	
In effect, by fixing $0<r\leq \rho$, we can find $j\in \mathbb{N}$ such that $\rho^{j+1}<r\leq \rho^{j}$.
Thus, combining \eqref{guina1} with \eqref{511} and \eqref{512}, we obtain
\begin{align*}
	\|u-\overline{l}\|_{L^{\infty}(B_{r})}&\leq \|u-\overline{l}\|_{L^{\infty}(B_{\rho^{j}})}\\&\leq \|u-{l_{j}}\|_{L^{\infty}(B_{\rho^{j}})}+\|l_{j}-\overline{l}\|_{L^{\infty}(B_{\rho^{j}})}\\
	&\leq  \rho^{j(1+\alpha)}+|a_{j}-\overline{a}|+\rho^{j}|b_{j}-\overline{b}|\\
	&\leq \rho^{j(1+\alpha)}\left(1+\frac{2C}{1-\rho^{\alpha}}\right)\\
	&\leq \frac{1}{\rho^{1+\alpha}}\left(1+\frac{2C}{1-\rho^{\alpha}}\right)r^{1+\alpha}.
\end{align*}
The proof is concluded.
\end{proof}

Finally, we are in a position to complete the proof of Theorem \ref{main}.

\begin{proof}[Proof of Theorem \ref{main}] The proof is divided into two cases.\\
{\bf Case 1.} $i(\Phi)\geq 0$. By the smallness regime, we may suppose that the assumption \eqref{5.6} of Corollary \ref{small11} is satisfied.
From Corollary \ref{small11}, we have known that $u$ is $C^{1,\alpha}$ at $0$. This is enough, in fact by standard translation
arguments we can prove that $u$ is also $C^{1,\alpha}$ for any point of $B_{1/2}$, thus getting that $u\in C^{1,\alpha}(B_{1/2})$. Consequently, we deduce the result of Theorem \ref{main} by Remark \ref{zhu2.5} and a covering argument.\\
{\bf Case 2.} $-1<i(\Phi)<0$. %By the smallness regime, we assume \eqref{small}. 
With the aid of Proposition \ref{qiyizhuantuihua}, we see that $u$ is a viscosity solution of the following equation
\begin{equation*}
	\hat{\Phi}(\abs{Du}, x)F(D^2 u, x)+\hat{H}(Du,x)	
	=\hat{f}(x) \quad  \text{in} \quad B_{1},
\end{equation*}
where
\begin{align*}
	\hat{\Phi}(\abs{t},x):=&\abs{t}^{-i(\Phi)}{\Phi}(\abs{t},x),\quad
	\hat{H}(t,x):=\abs{t}^{-i(\Phi)}{H}(t,x),\quad
	\hat{f}(x):=\abs{Du}^{-i(\Phi)}f(x).
\end{align*}
Note that $\hat{\Phi}$ satisfies the properties that the map $t\mapsto {\hat{\Phi}(t,x)}$ is almost non-decreasing, the map $t\mapsto \frac{\hat{\Phi}(t,x)}{t^{s(\Phi)-i(\Phi)}}$ is almost non-increasing with the same constant
$L\geq 1$, and $\hat{\Phi}(1,x)=1$ for all $x\in B_{1}$. Moreover, in view of \eqref{A4} and $0<-i(\Phi)<m-i(\Phi)\leq 1$,  we have
$$|	\hat{H}(t,x)|\leq |t|^{-i(\Phi)}\left(\mathcal{K}+\mathcal{M}|t|^{m}\right)\leq \hat{\mathcal{K}}+\hat{\mathcal{M}}|t|^{m-i(\Phi)}\quad {\rm for\;all} \;(t,x)\in \rn\times B_{1},$$
where $\hat{\mathcal{K}}:=\mathcal{K}$ and $\hat{\mathcal{M}}:=\mathcal{K}+\mathcal{M}$. Now we apply (iii) of Proposition \ref{jin2} to obtain
\begin{equation}\label{lip1}
	[u]_{C^{0,1}(B_{1/2})}\leq C
\end{equation}
for a universal constant $C>0$. %$C=C(d,\lambda,\Lambda,L,m,i(\Phi),C_{F},\theta)$>0.
%Since that $u$ is a Lipschitz continuous function, and 
As a consequence, the gradient $Du$ is bounded almost everywhere. Then we can estimate
\begin{equation*}
	\|\hat{f}\|_{L^{\infty}\left(B_{1/2}\right)}\leq C^{-i(\Phi)}\|{f}\|_{L^{\infty}\left(B_{1/2}\right)}. %\quad {\rm and}\quad
	%\|\hat{h}\|_{L^{\infty}\left(B_{1/2}\right)}\leq C^{-i(\Phi)}\delta.
\end{equation*}
%$$|	\hat{H}(Du,x)|\leq C^{-i(\Phi)}\left(\mathcal{K}+\mathcal{M}|Du|^{m}\right)=:\hat{\mathcal{K}}+\hat{\mathcal{M}}|t|^{m}.$$
 %Combining \eqref{lip1} with $\|{f}\|_{L^{\infty}\left(B_{1}\right)}\leq \delta$ and $\mathcal{K},\mathcal{M}\leq \delta$, we arrive at
%\begin{equation*}
%	\|\hat{f}\|_{L^{\infty}\left(B_{1/2}\right)},\hat{\mathcal{K}},\hat{\mathcal{M}}\leq C^{-i(\Phi)}\delta. %\quad {\rm and}\quad
	%\|\hat{h}\|_{L^{\infty}\left(B_{1/2}\right)}\leq C^{-i(\Phi)}\delta.
%\end{equation*}
At this point, we reduce the singular case to the degenerate case, and $\hat{\Phi}$, $\hat{H}$ and $\hat{f}$ satisfy the assumptions \eqref{A3}-\eqref{A5}. Therefore, for $-1<i(\Phi)<0$, we can obtain the $C^{1,\alpha}$-regularity with $\alpha$ satisfying \eqref{exponent} by repeating the previous arguments. The proof is complete.
\end{proof}
%%%%%%%%%%%%%%%%%%%%%%%%%%%%%%%%%%%%%%%%%%%%%%%%%%%%%%%%%%%%%%%%%%%%%%%%%%%%%%%%%%%%%%%%%%%%%%%%%%%%%%%%%%%%%%%
\section*{Acknowledgments}
This work is supported by the National Natural Science Foundation of China (NSFC Grant No.12571103), Natural Science Foundation of Tianjin (Grant No. 25JCQNJC01400), and Young Scientific and Technological Talents (Level Three) in Tianjin.
%%%%%%%%%%%%%%%%%%%%%%%%%%%%%%%%%%%%%%%%
\section*{Data availability} Data sharing is not applicable to this article as obviously no datasets were generated or analyzed during the current study.
%%%%%%%%%%%%%%%%%%%%%%%%%%%%%%%%%%%%%%
\section*{Conflict of interest} Author states no conflict of interest.
%%%%%%%%%%%%%%%%%%%%%%%%%%%%%%%%%%%%%%%


\begin{thebibliography}{99}
\bibitem{Adams2003}
R.A. Adams, J.J.F. Fournier, Sobolev Spaces, 2nd ed. Pure Appl. Math. vol.140, Elsevier/Academic Press,Amsterdam, NewYork, (2003).

\bibitem{Ricarte}
D.J. Ara$\acute{\rm u}$jo, G.C. Ricarte, E.V. Teixeira, Geometric gradient estimates for solutions to degenerate elliptic equations, Calc. Var. Partial Differ. Equ. 53(3–4) (2015), 605–625. 	

\bibitem{Attouchi20147}
A. Attouchi, M. Parviainen, E. Ruosteenoja, $C^{1,\alpha}$-regularity for the normalized $p$-Poisson problem, J. Math. Pures. Appl. 108 (2017), 553–591.

\bibitem{Andrade}
P. Andrade, T. Nascimento, Optimal regularity for degenerate elliptic equations with Hamiltonian
terms, arXiv:2508.03924.

\bibitem{Birindelli2004}
I. Birindelli, F. Demengel, Comparison principle and Liouville type results for singular fully nonlinear operators, Ann. Fac. Sci. Toulouse Math. 13(6)
(2004), 261–287.

\bibitem{Birindelli2006}
I. Birindelli, F. Demengel, First eigenvalue and maximum
principle for fully nonlinear singular operators. Adv. Differential Equations, 11(1) (2006), 91–119.

\bibitem{Birindelli2007CPAA}
I. Birindelli, F. Demengel, Eigenvalue, maximum principle and regularity for fully nonlinear homogeneous operators, Commun. Pure Appl. Anal. 6 (2007), 335–366.

\bibitem{Birindelli2010JDE}
I. Birindelli, F. Demengel, Regularity and uniqueness of the first eigenfunction for singular fully
nonlinear operators, J. Differ. Equations 249 (2010), 1089–1110.

\bibitem{Birindelli2012NON}
I. Birindelli, F. Demengel, Regularity for radial solutions of degenerate fully nonlinear equations, Nonlinear Anal. 75(17) (2012),
6237–6249.

\bibitem{Birindelli2014ESAIM}
I. Birindelli, F. Demengel,
$C^{1, \beta}$ regularity for Dirichlet problems associated to fully nonlinear degenerate elliptic equations, ESAIM Control Optim. Calc. Var. 20 (2014), 1009-1024.

\bibitem{Birindelli2015}
I. Birindelli, F. Demengel, H\"{o}lder regularity of the gradient for solutions of fully nonlinear equations with sublinear first order term. In Geometric methods in PDE's, pp. 257-268, Springer INdAM Ser. 13, Springer, Cham, (2015).

\bibitem{B-Demengel2016}
I. Birindelli, F. Demengel, Fully nonlinear operators with Hamiltonian: H\"{o}lder regularity of the gradient,
NoDEA Nonlinear Differ. Equ. Appl. 23 (4) (2016), 17 pp.

\bibitem{B-Demengel2019}
I. Birindelli, F. Demengel, F. Leoni, $C^{1,\gamma}$ regularity for singular or degenerate fully nonlinear equations and applications, NoDEA Nonlinear Differ. Equ. Appl. 26(5) (2019), 13 pp.

\bibitem{Bronzi2020}
A.C. Bronzi, E.A. Pimentel, G.C. Rampasso, E.V. Teixeira, Regularity of solutions to a class of variable exponent fully nonlinear elliptic equations. J. Funct. Anal. 279(12) (2020), 108781. 	

\bibitem{Banerjee2020}
A. Banerjee, I.H. Munive, Gradient continuity estimates for the normalized $p$-Poisson equation, Commun. Contemp. Math. 22 (2020).	

\bibitem{Silva2023}
E.C. Bezerra J$\acute{\rm u}$nior, J.V. da Silva, G.C. Rampasso, G.C. Ricarte, Global regularity for a class of fully nonlinear PDEs with unbalanced variable degeneracy, J. Lond. Math. Soc. 108(2) (2023), 622–665.

\bibitem{Baasandorj}
S. Baasandorj, S.-S. Byun, K.-A. Lee, S.-C. Lee, $C^{1,\alpha}$-regularity for a class of degenerate/singular fully nonlinear elliptic equations, Interfaces Free Bound. 26(2) (2024), 189–215.

\bibitem{Byun2024}
S. Baasandorj, S.-S. Byun, J. Oh, Second derivative $L^{\delta}$-estimates for a class of singular fully nonlinear elliptic equations, Nonlinear Anal. 249 (2024), 113630.

\bibitem{Byun2025}
 S.-S. Byun, H. Kim, J. Oh, Interior $W^{2,\delta}$ type estimates for degenerate fully nonlinear elliptic equations with $L^{n}$ data, J. Funct. Anal. 289(6) (2025), 111007.
	
	
\bibitem{Caffarelli1989}	
L.A. Caffarelli, Interior a priori estimates for solutions of fully nonlinear equations, Ann. Math. 130(1) (1989), 189–213.		

\bibitem{Crandle1}
M.G. Crandall. H. Ishii, P.-L. Lions, User's guide to viscosity solutions of second order partial differential equations, Bull. Am. Math. Soc. 27 (1992), 1–67.

\bibitem{Caff1}
L.A. Caffarelli, X. Cabr$\acute{\rm e}$, Fully Nonlinear Elliptic Equations, Colloquium Publications, 43. American Mathematical Society, Providence, R.I. (1995).

\bibitem{Leoni} 	
I. Capuzzo Dolcetta, F. Leoni, A. Porretta, H\"{o}lder estimates for degenerate elliptic equations with coercive Hamiltonians, Trans. Am. Math. Soc. 362(9) (2010), 4511–4536.

\bibitem{Davil2009}
J. Davil$\acute{\rm a}$, P. Felmer, A. Quaas, Alexandroff-Bakelman-Pucci estimate for singular or degenerate fully nonlinear elliptic equations, C. R. Math. Acad. Sci. Paris 347 (2009), 1165–1168.

\bibitem{Davil2010}
J. Davil$\acute{\rm a}$, P. Felmer, A. Quaas,
Harnack inequality for singular fully nonlinear operators and some existence results, Calc. Var. Partial Differ. Equ. 39 (2010), 557–578.

\bibitem{Silva2020}
J.V. da Silva, G.C. Ricarte, Geometric gradient estimates for fully nonlinear models with nonhomogeneous degeneracy and applications, Calc. Var. Partial Differ. Equ. 59(5) (2020), Paper No.161.

\bibitem{Silva-Nornberg2021}
J.V. da Silva,  G. Nornberg, Regularity estimates for fully nonlinear elliptic PDEs with general Hamiltonian terms and unbounded ingredients, Calc. Var. Partial Differ. Equ. 60 (2021), Paper No. 202, 40 pp.

\bibitem{Fili}
C. De Filippis, Regularity for solutions of fully nonlinear elliptic equations with nonhomogeneous degeneracy, Proc. R. Soc. Edinb. Sect. A 151(1)  (2021), 110–132.

\bibitem{Evans}
L.C. Evans, Classical solutions of fully nonlinear, convex, second-order elliptic equations, Commun.
Pure Appl. Math. 35(3) (1982), 333–363.

\bibitem{Fleming} 	
W. Fleming, H.M. Soner, Controlled Markov Processes and Viscosity Solutions, Applications of Mathematics 25,
Springer-Verlag, (1991).

\bibitem{Fang}
Y. Fang, V.D. R$\breve{\rm a}$dulescu, C. Zhang, Regularity of solutions to degenerate fully nonlinear elliptic equations with variable exponent, Bull. Lond. Math. Soc. 53(6) (2021), 1863–1878.

\bibitem{Lions1}
H. Ishii, P.-L. Lions, Viscosity solutions of Fully-Nonlinear Second Order Elliptic Partial Differential Equations, J. Differ. Equations 83 (1990), 26–78.

\bibitem{Imbert2011JDE}
C. Imbert, Alexandroff-Bakelman-Pucci estimate and Harnack inequality for degenerate/singular fully
non-linear elliptic equations, J. Differ. Equations 250 (2011), 1553–1574.	

\bibitem{Imbert1}
C. Imbert, L. Silvestre, $C^{1,\alpha}$ regularity of solutions of some degenerate fully non-linear elliptic equations, Adv. Math. 233 (2013), 196–206.

\bibitem{Junges2010}
T. Junges Miotto, The Aleksandrov-Bakelman-Pucci estimates for singular fully nonlinear operators,
Commun. Contemp. Math. 12(4) (2010), 607–627.

\bibitem{Krylov1979}
N.V. Krylov, M.V. Safonov, An estimate for the probability of a diffusion process hitting a set of positive measure, Dokl. Akad. Nauk SSSR 245(1) (1979), 18–20.

\bibitem{Krylov1980}
N.V. Krylov,  M.V. Safonov, A property of the solutions of parabolic equations with measurable coefficients, Izv. Akad. Nauk SSSR Ser. Mat. 44(1) (1980), 161–175.

\bibitem{Krylov1}
N.V. Krylov, Boundedly inhomogeneous elliptic and parabolic equations, Izv. Akad. Nauk SSSR Ser.
Mat. 46(3) (1982), 487–523.

\bibitem{Lions1985}
P.-L. Lions, Quelques remarques sur les problemes elliptiques quasilineaires du second ordre, J. Anal. Math. 45 (1985), 234–254.

\bibitem{Lions1989} 		
J.M. Lasry, P.-L. Lions, Nonlinear elliptic equations with singular boundary conditions and stochastic control with state constraints. I. The model problem, Math. Ann. 283(4) (1989), 583–630.	

\bibitem{Nornberg2019}
G. Nornberg, $C^{1,\alpha}$ regularity for fully nonlinear elliptic equations with superlinear growth in the gradient, J. Math. Pures Appl. 9(128) (2019), 297-329.
\end{thebibliography}
\end{document}